\documentclass[10pt,leqno,twoside]{amsart}

\usepackage{amsmath}
\usepackage{amsthm}
\usepackage{amssymb}
\usepackage{amsfonts}
\usepackage{mathtools}
\usepackage{mathrsfs}
\usepackage{cite}
\usepackage{dsfont}
\usepackage{hyperref}
\usepackage{inputenc}
\usepackage{tikz, subfigure, xcolor} 
\usetikzlibrary{snakes}
\usepackage[normalem]{ulem}
\usepackage{doi}
\usepackage{comment}

\newtheorem{theorem}{Theorem}[section]
\newtheorem{lemma}[theorem]{Lemma}
\newtheorem{proposition}[theorem]{Proposition}
\newtheorem{corollary}[theorem]{Corollary}
\theoremstyle{definition}
\newtheorem{definition}[theorem]{Definition}
\newtheorem{remark}[theorem]{Remark}

\newtheorem{example}[theorem]{Example}

\renewcommand{\div}{\mathrm{div}\,}				

\DeclareMathOperator{\supp}{\mathrm{supp}\,}

\renewcommand{\d}{\operatorname{d}\!} 		
\DeclareMathOperator{\ds}{\d s}				

\newcommand{\R}{\mathbb{R}}

\newcommand{\N}{\mathbb{N}}
\newcommand{\E}{\mathbb{E}} 
\newcommand{\T}{\mathbb{T}} 

\newcommand{\triple}{|\!|\!|}
\newcommand{\norm}[1]{\left\lVert#1\right\rVert} 
\newcommand{\abs}[1]{\left\vert#1\right\vert} 

\makeatletter
\newcommand\RSloop{\@ifnextchar\bgroup\RSloopa\RSloopb}
\makeatother
\newcommand\RSloopa[1]{\bgroup\RSloop#1\relax\egroup\RSloop}
\newcommand\RSloopb[1]%
  {\ifx\relax#1%
   \else
     \ifcsname RS:#1\endcsname
       \csname RS:#1\endcsname
     \else
       \GenericError{(RS)}{RS Error: operator #1 undefined}{}{}%
     \fi
   \expandafter\RSloop
   \fi
  }
\newcommand\X{0}
\newcommand\RS[1]%
  {\begin{tikzpicture}
     [every node/.style=
       {circle,draw,fill,minimum size=1.5pt,inner sep=0pt,outer sep=0pt},
      line cap=round
     ]
   \coordinate(\X) at (0,0);
   \RSloop{#1}\relax
   \end{tikzpicture}
  }
\newcommand\RSdef[1]{\expandafter\def\csname RS:#1\endcsname}
\newlength\RSu
\RSdef{i}{\draw (\X) -- +(90:\RSu) node{};}
\RSdef{l}{\draw (\X) -- +(135:\RSu) node{};}
\RSdef{r}{\draw (\X) -- +(45:\RSu) node{};}
\RSdef{I}{\draw (\X) -- +(90:\RSu) coordinate(\X I);\edef\X{\X I}}
\RSdef{L}{\draw (\X) -- +(135:\RSu) coordinate(\X L);\edef\X{\X L}}
\RSdef{R}{\draw (\X) -- +(45:\RSu) coordinate(\X R);\edef\X{\X R}}
\RSdef{1}{\draw[snake=zigzag, segment length=1.25pt, segment amplitude=0.3pt] (\X) -- +(135:\RSu) coordinate(\X L);\edef\X{\X L}}
\RSdef{2}{\draw[snake=zigzag, segment length=1.25pt, segment amplitude=0.3pt] (\X) -- +(90:\RSu) coordinate(\X I);\edef\X{\X I}}
\RSdef{3}{\draw[snake=zigzag, segment length=1.25pt, segment amplitude=0.3pt] (\X) -- +(45:\RSu) coordinate(\X R);\edef\X{\X R}}
\RSdef{4}{\draw[thick, snake=ticks, segment length=1.5pt, segment amplitude=0.1pt] (\X) -- +(135:\RSu) coordinate(\X L);\edef\X{\X L}}
\RSdef{5}{\draw[thick, snake=ticks, segment length=1.5pt, segment amplitude=0.1pt] (\X) -- +(90:\RSu) coordinate(\X I);\edef\X{\X I}}
\RSdef{6}{\draw[thick, snake=ticks, segment length=1.5pt, segment amplitude=0.1pt] (\X) -- +(45:\RSu) coordinate(\X R);\edef\X{\X R}}

\title{Surface Quasi-Geostrophic Equation driven by Space-Time White Noise}



\author[Forstner]{Philipp Forstner}
\address{Technische Universit\"at Berlin, Berlin, Germany}
\email{forstner@math.tu-berlin.de}
\curraddr{Straße des 17. Juni 136, 10587 Berlin, Germany}

\author[Saal]{Martin Saal}
\address{Scuola Normale Superiore,
	Pisa, Italy}
\email[Corresponding author]{msaal@mathematik.tu-darmstadt.de}
\curraddr{TU Darmstadt, Schlossgartenstr. 7, 64289 Darmstadt, Germany}

\thanks{The first author gratefully acknowledges the financial support of the Deutsche Forschungsgemeinschaft (DFG) through IRTG 2544.  The second author acknowledges the financial support of the DFG through the research fellowship SA 3887/1-1.
}

\begin{document}

\begin{abstract}
We consider the Surface Quasi-Geostrophic equation (SQG) driven by space-time white noise and show the existence of a local in time solution by applying the theory of regularity structures. A main difficulty is the presence of Riesz-transforms in the non-linearity. We show how to lift singular integral operators with a particular structure to the level of regularity structures and using this result we deduce the existence of a solution to SQG by a renormalisation procedure. The fact that the Riesz-transforms act only in the spatial variable makes it necessary to use inhomogeneous models instead of the standard ones.
\end{abstract}

\maketitle


\section{Introduction}
The dissipative Surface Quasi-Gestrophic equation (SQG) in the torus $\mathbb{T}=\mathbb{R}^2\slash\mathbb{Z}^2$ with forcing $f$ is given by
\begin{equation*}
	\begin{split}
		& \partial_{t}\theta+u\cdot\nabla\theta + \nu (-\Delta)^{\mu}\theta =f,\\
		& u   =\nabla^{\perp}(-\Delta)^{-1/2}\theta=R^{\perp}\theta,
	\end{split}
\end{equation*}
where $\mu\in (0,1]$ and $\nu>0$.
Here, $R=(R_1, 
R_2)$ is the vector of Riesz-transforms,
$\nabla^{\perp}=(-\partial_{y},\partial_{x})$ and $u\cdot\nabla=u_1\partial_1+u_2\partial_2$. It has applications in both 
meteorological and oceanic flows, and it describes the temperature $\theta$ in 
a rapidly rotating stratified fluid with uniform potential vorticity. In 
mathematics, it received a lot of attention because of structural similarities 
to the 3D Navier-Stokes equations as pointed out in \cite{ConstantinMajdaTabak1} 
and \cite{CMT_1994},  e.g. vortex stretching. For simplicity, we normalize $\nu=1$.

For $\mu=1/2$ the equation has a geophysical application, it describes 
the evolution of the temperature on the 2D boundary of a rapidly 
rotating half-space with small Rossby and Ekman numbers, the dissipative 
term $(-\Delta)^{1/2}$ models the Ekman-pumping (\!\cite{Constantin2}). 
Dimensionally, it is the analogue of the 3D Navier-Stokes equations 
(\!\cite{ConstantinWu}). For more information on the geophysical background see \cite{HeldPierreGarnerSwanson, Pedlosky1987, Smithetal}.
There are several reasons to investigate stochastic aspects of PDEs with a background in fluid 
dynamics. Stochastics can on the one hand be used to account for numerical and empirical uncertainties 
and thus provide a way to study the robustness of a basic model. On the other hand complex phenomena 
related to turbulence may also be produced by stochastic perturbations and 
there are examples for intrinsic stochasticity within the physical model, see \cite{Hasselmann1976, 
	Holm2015, Memin2014, FrederiksenKaneZidikheri13, CrisanFlandoliHolm2018} and the 
references therein.  To take such 
effects into account, we will consider the SQG equation driven by space-time white noise.

Let us first give an overview of existence and uniqueness results for the system, both in the deterministic and stochastic setting.
Global weak solutions for 
initial data $\theta_0 \in L^2(\mathbb{T}^2)$ and $f=0$ exist for the full range $\mu\in (0,1)$ (\!\!\cite{Resnick}). Also, for initial data $\theta_0\in L^p(\mathbb{T}^2)$ 
($4/3 \leq p <\infty$) there are global weak solutions and one can go down to initial data $\theta_0\in H^{-1/2}(\mathbb{T}^2)$ and obtain 
corresponding global weak solutions (\!\!\cite{Marchand}). The uniqueness of these weak 
solutions is an open problem. 
The non-uniqueness of weak solutions 
with less regularity than the above ones is established by the technique of 
convex integration in \cite{BuckmasterShkollerVicol}.

For $\mu>1/2$ and
initial data $\theta_0\in H^{\beta}(\mathbb{T}^2)$ with $\beta+\mu>2$ strong 
solutions exist globally (\!\!\cite{ConstantinWu}).
The global regularity issue for $\mu=1/2$ was solved by different methods 
in \cite{CaffarelliVasseur2010} and \cite{KiselevNazarovVolberg2007}, for initial data in 
$L^2(\mathbb{T}^2)$ the solution is smooth for all times $t>0$ and exists globally in time.
For further results on the deterministic dissipative SQG equation and also for the inviscid SQG equation see \cite{Constantin} and the references therein. 

The stochastic SQG equation has been studied extensively in \cite{RZZ_2015}
with additive and multiplicative noise. They show for 
$\mu\in(0,1)$ the existence of weak solutions for multiplicative noise and the uniqueness and 
existence of pathwise strong solutions for $\mu>1/2$. Additionally, for additive noise and 
$\mu>1/2$ ergodicity, a law of large numbers and the polynomial convergence of the corresponding 
Markov semigroup to the invariant measure is established. If $\mu>\frac23$, they improve the result 
to exponential convergence by showing that not only the asymptotical strong Feller property but the 
strong Feller property itself holds. For $\mu>1/2$ large deviation principles are shown for small 
multiplicative noise and small times in \cite{LiuRoecknerZhu2013}, and the existence of a random attractor 
is proven if the noise is additive or linear multiplicative in \cite{ZhuZhu2017}.
In these works the noise is assumed to be a cylindrical Wiener process over some Hilbert space.

We will consider here the case of additive space-time white noise $\xi$, i.e. the equation
\begin{equation}\label{IntroductionSQG}
		\partial_{t}\theta+(R^\perp \theta)\cdot\nabla\theta +  (-\Delta)^{\mu}\theta =\xi.
\end{equation}

The methods of the previous works on the stochastic SQG cannot be applied to our situation due to the roughness of the noise. The Riesz-transforms leave the regularity of the function it is applied to invariant, see \cite{GRA_2014}, hence, a solution is not regular enough to give a meaning to the nonlinearity as the product of a distribution and a function in the classical way since both factors have negative regularity. Instead, the theory of regularity structures leads to the following result. For the definition of the function spaces see Subsection \ref{ss:notaion}.

\begin{theorem}\label{theorem:maintheorem}
	Let $\mu>2/3$,  $\eta\in(-2\mu+1,0)$, $\varepsilon>0$ and $\rho$ be a symmetric mollifier, $\rho_{\varepsilon}:=\rho(\cdot/\varepsilon)$ and $\xi_{\varepsilon}=\rho_{\varepsilon}\ast \xi$. Then for any initial condition $\theta_0\in C^\eta$ there is a $T>0$ such that the sequence of solutions $\theta_\varepsilon$ to 
	\begin{align*}
	\partial_{t}\theta_\varepsilon+(R^\perp \theta_\varepsilon)\cdot\nabla\theta_\varepsilon +  (-\Delta)^{\mu}\theta_\varepsilon =\xi_{\varepsilon},
	\end{align*}
	$\theta_{\varepsilon}(0)=\theta_0$ can be renormalised and the renormalized solutions $\hat{\theta}_{\varepsilon}$ converge in probability to a limit $\theta\in C^{\bar{\delta},-2+2\mu-2\kappa}_{\eta-1+\mu-2\kappa;T}$ ($\varepsilon\to 0$), where $0<\bar{\delta}<3\mu-2-\kappa$ for $\kappa$ sufficiently small.
	The limit $\theta$ is independent of the particular choice of the mollifier $\rho$.
\end{theorem}

In general, the theory of regularity structures works as follows, see \cite{FriHai14} and \cite{Hairer_2014} for an introduction to the topic. 
First, in an algebraic step, one builds a regularity structure, associated to the equation at hand, rich enough to reformulate the equation as an abstract fixed point problem. Second, in an analytical step, the abstract fixed point problem will be solved using classical analysis theory like the Banach fixed point theorem. Third, in a probabilistic step, one reconstructs solutions with respect to more regular instances of the driving noise and ensures the convergence of the reconstructed solutions when removing the regularisation. The convergence cannot be expected for the sequence itself if the driving noise is rough, in these cases one needs to renormalise the solutions in order to obtain a limiting process.

This approach leads to a local in time solution for subcritical equations under mild assumptions on the nonlinearity which is shown in a series of papers  \cite{BrunedChandraChevyrevHairer2021}, \cite{BrunedHairerZambotti2019} and \cite{ChaHairer_2016}, where a black box theorem \cite[Theorem 2.22]{BrunedChandraChevyrevHairer2021} for local well-posedness of SPDEs is finally obtained. Subcritical means in this context that the smoothing of the linear part is stronger than the loss of regularity by the nonlinearity. For SQG it is easy to check that one needs  $\mu>2/3$ to guarantee this property.

However, before applying those results to SQG we first need to define the abstract version of the Riesz-transform, so that the analytic step can be carried out. Our result is Theorem \ref{TheoremAbstractIntegrationR} and it will turn out, that the lift of the Riesz-transform is a continuous (and linear) mapping on the function spaces considered to find solutions on the abstract level.

We face two difficulties compared to the lift of smoothing kernels carried out in \cite{Hairer_2014}. Firstly, we want to lift a convolution only in the spatial variables, which means that it is formally a convolution with the Dirac distribution in time and so the kernel is not regular enough to carry over the approach.

Secondly, the kernel of the Riesz-transform has a polynomial divergence at the origin whose order equals the spatial dimension, and so certain non-convergent series appear when following directly the ideas of \cite{Hairer_2014}.

To work around the first problem, we use inhomogeneous models introduced in \cite{hairer2017discretisations} where the time and spatial variables are stronger separated than in the standard setting. The disadvantage here is, that we need to be able to interpret all terms appearing in a representation of the solution as functions in time (and distributional valued). It will turn out, that in our case, this is no restriction but in general it can lead to stronger regularity assumptions compared to the standard models. Another approach we want to mention was applied in \cite{BerglundKuehn2019} to the Allen-Cahn equation where an integration in time has to be lifted to take into account a Fitz-Hugh Nagumo type equation. Here, the key observation is that not the non-smoothing kernel itself has to be lifted but that it is sufficient to lift the composition of the kernel describing the linear smoothing with this non-smoothing one. In this case, the standard models can be used, but the combined kernel has a singularity along the $t=0$ hyperplane. That approach is not convenient if the dimension of the hyperplane is large, so we do not follow it to lift a convolution acting in all spatial variables.

To tackle the second problem we need more structure of the kernel. Our proofs are based on the assumption that the associated kernel can be written as the derivative of a smoothing kernel. This allows us to shift derivatives from the kernel to test-functions. The fact, that we need more structure than just a certain scaling behaviour near the singularity is not only a technical point. One cannot expect to be able to lift any kernel that behaves like $|x|^{-d}$, since these kernels do not induce a well-behaved operator, see Example \ref{example}. 

Examples of operators which can be lifted by our method include the Riesz-transforms, the Helmholtz resp. Leray projection and also the Bogovskii-operator.

Having lifts of all analytic operations at hand to formulate SQG on the level of regularity structures, we then turn to \eqref{IntroductionSQG}. We construct the regularity structure associated to SQG and perform the analytic step to show that the lifted equation has a unique solution for all $\mu>2/3$ in Theorem \ref{theprem:abstractsolvable}.
The nonlinearity has certain symmetries leading to cancellations, which helps to handle it in spaces of low regularity (see e.g.  \cite{FlandoliSaal, BuckmasterShkollerVicol}). We carry out the renomralisation in detail for $\mu>4/5$ on the one hand to provide a self-contained result for readers not familiar with \cite{BrunedChandraChevyrevHairer2021} and on the other hand to show that in this case no renormalisation constants in the approximating equations appear when carrying out the limit process. 
The local existence theorem is a first step towards the treatment of the critical case $\mu=2/3$. A similar attempt to obtain information about the critical case by taking out the regularization of the linear part has been done in \cite{BerglundKuehn2017}. There a fractional heat equation with polynomial nonlineartiy is assumed and they can proof a number of results on the corresponding model space when approaching the critical case. Anyhow, in their nonlinearity no such cancellation as in SQG appear, so our situation is more promising to obtain stronger results in a forthcoming work.
Note that for $\mu=1$ the situation is similar to the 2D Navier-Stokes equation and one can proceed as in \cite{DaPratoDebussche2002} by the Da Prato-Debuschee argument.

In the context of fluid dynamics regularity structures have been applied to the three-dimensional Navier-Stokes equations in \cite{ZhuZhu_2014_2}. Note that this preprint differs substantially from the published version \cite{ZhuZhu_2014}, where the Helmholtz projection has not been handled. Anyhow, the results in \cite{ZhuZhu_2014_2} on the lift of non-smoothing operators are too strong, the authors claim that any operator which is 0-smoothing in the sense of \cite{Hairer_2014} can be lifted. Hence, we don't use their results to lift the Riesz-transform.

This paper is organized as follows. In the following Subsection \ref{ss:notaion} we list the notation used throughout this work. In Section \ref{s:nonsmoothingkernels} we give a brief introduction into regularity structures with inhomogeneous models and we show how non-smoothing singular integral operators can be lifted to the level of regularity structures.
In Section \ref{s:sqg} we apply that result to construct the regularity structure for SQG and to obtain the local well-posedness of the lift \eqref{IntroductionSQG} and give some details on the renormalisation.

\subsection{Notation}\label{ss:notaion}
In Section 2  we will work in $\R^{1+d}$, with $d$ being the spatial dimensions, for the application to SQG in Section 3 we then have $d=2$. For $s_0\in\N$ we consider the scaling $\mathfrak s:=(s_0,1,\dots,1)$ on $\R^{1+d}$ and set $|\mathfrak{s}|=s_0+d$. For $x\in\R^d$ we denote the norm by $|x|:=\sup_{1\leq i \leq d}|x_i|$ and for $(t,x)=(t,x_1,\dots,x_d)\in\R^{1+d}$ we define the scaled norm $\|(t,x)\|_{\mathfrak s}=|t|^{1/s_0} \lor |x|$. For $k=(k_0,\dots,k_d)\in \N_0^{1+d}$ we set $|k|_s:=s_0k_0+\sum_{i=1}^d k_i$. We use the 
multiindex notation $x^k=x_1^{k_1}\cdots x_d^{k_d}$ and  $(t,x)^k=t^{k_0} x_1^{k_1}\cdots x_d^{k_d}$ if $k\in \N^d_0$ or resp. in $\N_0^{d+1}$.\\
Furthermore, if $s,t\in\R$ we write $|s,t|_0:=|t|^{1/s_0}\land|s|^{1/s_0}\land 1 $ and $|t|_0:=|t,t|_0=|t|^{1/s_0} $.

By $\mathcal{S}'(\R^d)$ we denote as usual the space of all Schwartz distributions. Given an integer $r>0$, we define 
\begin{align*}
\mathcal{B}_r\coloneqq\{\varphi\in C^r(\R^d,\R): \norm{\varphi}_{C^r}\leq 1\text{ and }\supp{\varphi}\subset B_1(0)\}
\end{align*}
and for $\varphi\in \mathcal{B}_r$, $x\in\R^d$ and $\lambda\in(0,1]$, we define the scaled test-function $\varphi_x^\lambda$ by
\begin{align*}
\varphi_x^\lambda(y)\coloneqq\lambda^{-d}\varphi(\lambda^{-1}(y-x)) \qquad (y\in\R^d).
\end{align*}
The space of $\alpha$-H\"older continuous functions on $\R^d$ is for $\alpha\in (0,1)$ defined by
\begin{multline*}
C^\alpha(\R^d)\coloneqq\{f\in C^0(\R^d):  \abs{f(x)-f(y)}\leq C\abs{x-y}^\alpha
  \text{ uniformly in } x,y\in\R^d\}.
\end{multline*}
For $\alpha = n + \beta$ with $n\in\N_0$ and $\beta\in (0,1)$ the space $C^{\alpha}(\R^d)$ consists of all $n$-times continuously differentiable functions, such that all $n$-th derivatives of $f$ are in $C^{\beta}(\R^d)$.\\
If $\alpha <0$ let  $r>-\alpha$. We define for $\xi \in\mathcal{S}'(\R^d)\cap (C^r_c(\R^d))'$ where $(C^r_c(\R^d))'$ denotes the dual of all $C^r(\R^d)$ functions with compact support
\begin{align*}
\|\xi\|_{C^\alpha}:=\sup_{x\in\R^d} \sup_{\lambda\in(0,1]} \sup_{\varphi\in\mathcal{B}_r} \lambda^{-\alpha}\abs{\langle\xi,\varphi_x^\lambda\rangle}
\end{align*}
and set
\begin{align*}
C^{\alpha}(\R^d)\coloneqq\{\xi & \in\mathcal{S}'(\R^d)\cap (C^r_c(\R^d))' :\|\xi\|_{C^{\alpha}}<\infty \}.
\end{align*}
We write $C^{\alpha}$ for $C^{\alpha}(\R^d)$  to shorten the notation.\\
Finally, for a function $\xi:(0,T]\to C^{\alpha}$, where $T>0$, $\alpha\in\R\setminus\N_0$ we define for $\delta>0,\eta\leq 0$ with $\alpha-\delta\notin \N_0$ the weighted norm
\begin{align*}
\|\xi\|_{C^{\delta,\alpha}_{\eta;T}}:= \sup_{t\in(0,T]} |t|_0^{-\eta}\|\xi(t)\|_{C^{\alpha}}+\sup_{s\neq tt\in(0,T]} \frac{\|\xi(t)-\xi(s)\|_{C^{\alpha-\delta}}}{|t,s|_0^{\eta} |t-s|^{\delta/s_0}}
\end{align*}
and 
\begin{align*}
	C^{\delta,\alpha}_{\eta;T}\coloneqq\{\xi:(0,T]\to C^{\alpha}  :\|\xi\|_{C^{\delta,\alpha}_{\eta;T}}<\infty \}.
\end{align*}
For the evaluation of a function at some $t$ we will use both notations $\xi_t=\xi(t)$.\\

For space-time test-functions for which we reserve the notation $\tilde{\varphi}$. For $(t,x)\in\R^{1+d}$ the scaled test-functions are then 
\begin{align*}
\tilde{\varphi}_{t,x}^\lambda(s,y)\coloneqq\lambda^{-s_0-d}\tilde{\varphi}(\lambda^{-s_0}(t-s),\lambda^{-1}(y-x)) \qquad (s\in\R,y\in\R^d)
\end{align*}
and define the corresponding space $C^\alpha$ as above. The convolution in time and space we denote by $\ast$ and the convolution only in the spatial variable by $\ast_{\R^d}$.

\section{Regularity structures and non-smoothing kernels}\label{s:nonsmoothingkernels}
In this section we will give the definitions and results in the case that the domain is $\R^{1+d}$, but all the results also hold in the torus.

\begin{definition}
A \textit{regularity structure} $\mathscr{T}=(A,\mathcal{T},G)$ consists of the following elements:
\begin{itemize}
\item An index set $A\subset\R$ such that $0\in A$, $A$ is bounded from below and locally finite.
\item A \textit{model space} $\mathcal{T}$, which is a graded vector space $\mathcal{T}=\bigoplus_{\alpha\in A}\mathcal{T}_\alpha,$ with each $\mathcal{T}_\alpha$ a Banach space. Elements in $\mathcal{T}_\alpha$ are said to have \textit{homogeneity} $\alpha$. Furthermore, $\mathcal{T}_0=\langle {\color{blue}1}\rangle\cong\R$. Given ${\color{blue}\tau}\in T$, we will write $\norm{{\color{blue}\tau}}_\alpha$ for the norm of its component in $\mathcal{T}_\alpha$.
\item A \textit{structure group} $G$ of linear operators acting on $\mathcal{T}$ such that, for every $\Gamma\in G$, every $\alpha\in A$, and every ${\color{blue}{\tau_\alpha}}\in \mathcal{T}_\alpha$, one has
\begin{equation*}
\Gamma{\color{blue}\tau_\alpha}-{\color{blue}\tau_\alpha}\in \mathcal{T}_{<\alpha}\coloneqq\bigoplus_{\beta <\alpha} \mathcal{T}_\beta.
\end{equation*}
Furthermore, $\Gamma{\color{blue}1}={\color{blue}1}$ for every $\Gamma\in G$.
\end{itemize}
\end{definition}

For the neutral element in $G$ we will simply write $1$. One easy example is the polynomial regularity structure. We will need it later on, so let us give the 
precise definition.

\begin{definition}\label{Def:PolynomialRegularityStructure}
The \textit{polynomial regularity structure} $\mathscr{T}_P=(A,\mathcal{T},G)$ on $\R\times\R^{d}$ is given by:
	\begin{itemize}
		\item The set of homogeneities $A=\N_0$.
		\item The model space $\mathcal{T}=\bigoplus_{\alpha\in\N_0}\langle{\color{blue}X^k}:k\in\N_0^d,\abs{k}_s=\alpha\rangle=\R[{\color{blue}X_0},...,{\color{blue}X_d}]$.
		\item The structure group $G\coloneqq\{\Gamma_h\in\mathcal{L}(\mathcal{T},\mathcal{T}): h\in\R\times\R^d\}\sim (\R^{1+d},+)$ acting on $\mathcal{T}$ via 
		$$\Gamma_h P({\color{blue}X})\coloneqq P({\color{blue}X}+h{\color{blue}1}),\quad h\in\R\times\R^{d},$$
		for any polynomial $P$.
	\end{itemize}
\end{definition}

To associate a distribution to an abstract element $\tau \in \mathcal{T}$, so-called models are used. 

\begin{definition}
Given a regularity structure $\mathscr{T}=(A,\mathcal{T},G)$, an \textit{inhomogeneous model} $M=(\Pi,\Gamma, \Sigma)$ for $\mathscr{T}$ on $\R\times \R^{d}$ consists of maps
\begin{align*}
\Pi 	&:\R\times \R^d \rightarrow\mathcal{L}\left(\mathcal{T},\mathcal{S}'(\R^d)\right),\quad (t,x)\mapsto\Pi_x^t,\\
\Gamma 	&:\R\times \R^d\times\R^d\rightarrow G,\quad (t,x,y)\mapsto\Gamma_{xy}^t\\
\Sigma 	&:\R\times \R \times\R^d\rightarrow G,\quad (t,s,x)\mapsto\Sigma_{x}^{ts}
\end{align*}
such that for $x,y,z\in\R^d, t,s,r \in \R$
\begin{align*}
\Gamma_{xx}^{t}&=1, \qquad \Gamma_{xy}^{t}\Gamma_{yz}^{t}=\Gamma_{xz}^{t}\\
\Sigma_{x}^{tt}&=1, \qquad \Sigma_{x}^{sr}\Sigma_{x}^{rt}=\Sigma_{x}^{st}
\end{align*}
and
\begin{align*}
\Sigma_{x}^{st}\Gamma_{xy}^{t}=\Gamma_{xy}^{s}\Sigma_{y}^{st}\\
\Pi_x^t\Gamma_{xy}^{t}=\Pi_y^t.
\end{align*}
Furthermore, for every $T>0$ and every $\gamma>0$, there exists a constant $C=C(T,\gamma)$ such that the bounds
\begin{align}\label{ModelEstimates}
\begin{split}
\abs{\langle\Pi_x^t{\color{blue}\tau},\varphi_x^\lambda\rangle} & \leq C\lambda^\alpha\norm{{\color{blue}\tau}}_\alpha,\\
\quad \norm{\Gamma_{xy}^t{\color{blue}\tau}}_\beta & \leq C\abs{x-y}^{\alpha-\beta}\norm{{\color{blue}\tau}}_\alpha,\\
\quad \norm{\Sigma_{x}^{ts}{\color{blue}\tau}}_\beta & \leq C\abs{t-s}^{(\alpha-\beta)/s_0}\norm{{\color{blue}\tau}}_\alpha
\end{split}
\end{align}
hold uniformly over $\varphi\in\mathcal{B}_k,\ x,y\in\R^d, t,s\in[-T,T] , \lambda\in (0,1],\ {\color{blue}\tau}\in \mathcal{T}_\alpha$ with $\beta<\alpha\leq\gamma$ and $k>\abs{\min{A}}$.\\
We say that $\Pi$ has time regularity $\delta>0$ if additionally
\begin{align}\label{ModelEstimatesTimeRegularity}
\abs{\langle\Pi_x^t{\color{blue}\tau}-\Pi_x^t{\color{blue}\tau},\varphi_x^\lambda\rangle}\leq C \abs{t-s}^{\delta/s_0} \lambda^{\alpha-\delta}\norm{{\color{blue}\tau}}_\alpha
\end{align}
holds under the above restrictions.
\end{definition}

\begin{remark}\label{Rem:Model}
\begin{itemize}
\item[a)] The map $\Pi_x$ realizes abstract elements as a distribution with base point $(t,x)$ while the map $\Gamma_{xy}^t$ together with the identity $\Pi_x^t\Gamma_{xy}^t=\Pi_y^t$ reflects the turning of an expansion around $y$ into an expansion around $x$. Respectively, $\Sigma$ describes shifts in time.
\item[b)] If we have an inhomogeneous model $M=(\Pi,\Gamma, \Sigma)$ given we define $\tilde M:=(\tilde \Pi,\tilde \Gamma)$ by setting
\begin{align*}
(\tilde \Pi_{(t,x)}\tau)(\tilde\varphi)&:= \int_\R (\Pi^s_x\Sigma^{st}_x\tau)(\tilde \varphi(s,\cdot)) \d s,\\
\tilde \Gamma_{(t,x)(s,y)}&:= \Gamma^{t}_{xy}\Sigma^{ts}_{y}=\Sigma^{ts}_{x}\Gamma^{s}_{xy}
\end{align*}
for any test-function $\tilde\varphi:\R\times \R^d$ and $(t,x),(s,y)\in \R^{1+d}$. Then $\tilde M$ is a model in the sense of \cite{Hairer_2014}.
\item[c)] In the applications we will only need to consider a finite set of homogeneities bounded by some $\gamma>0$ and the corresponding subset $\mathcal{T}_{<\gamma}:=\oplus_{\alpha\in A,\alpha<\gamma} \mathcal{T}_\alpha$ of the model space $\mathcal{T}$.
\item[d)] For an inhomogeneous given model $M=(\Pi,\Gamma, \Sigma)$ on a regularity structure $\mathscr{T}$ and $\gamma>0$ we denote by $\|\Pi\|_{\gamma;T}, \|\Gamma\|_{\gamma;T}$ and $\|\Sigma\|_{\gamma;T}$ the smallest constants such that the bounds \eqref{ModelEstimates} hold. If the model has time regularity $\delta>0$ we denote by $\|\Pi\|_{\delta,\gamma;T}$ the smallest constant such that \eqref{ModelEstimatesTimeRegularity} holds. We set
\begin{align*}
\|M\|_{\gamma;T}&:= \|\Pi\|_{\gamma;T}+ \|\Gamma\|_{\gamma;T} + \|\Sigma\|_{\gamma;T},\\
\|M\|_{\delta,\gamma;T}&:= \|\Pi\|_{\delta,\gamma;T}+ \|M\|_{\gamma;T}.
\end{align*}
If $\overline M:=(\overline \Pi,\overline \Gamma,\overline \Sigma)$ is second model, then in general $(\Pi-\overline \Pi,\Gamma-\overline \Gamma,\Sigma-\overline \Sigma)$ is not a model, but the bounds \eqref{ModelEstimates} and \eqref{ModelEstimatesTimeRegularity} still make sense and we define
\begin{align*}
\|M;\overline M\|_{\gamma;T}&:= \|\Pi-\overline \Pi\|_{\gamma;T}+ \|\Gamma-\overline \Gamma\|_{\gamma;T} + \|\Sigma-\overline \Sigma\|_{\gamma;T},\\
\|M;\overline M\|_{\delta,\gamma;T}&:= \|\Pi-\overline \Pi\|_{\delta,\gamma;T}+ \|M-\overline M\|_{\gamma;T}.
\end{align*}
\end{itemize}
\end{remark}

With the definition of an inhomogeneous model at hand, we are now able to construct such an 
inhomogeneous model for the polynomial regularity structure $\mathscr{T}_P$ given in Definition \ref{Def:PolynomialRegularityStructure}. 

\begin{definition}
For $\mathscr{T}_P$ we define the inhomogeneous model $M=(\Pi,\Gamma, \Sigma)$ by 
\begin{align}\label{PolynomialRegularityStructureModel}
		\begin{split}
			\Pi_x^t{\color{blue}X}^k&:=\begin{cases} (x-y)^{\overline k}, &k_0=0 \\ 0 , & k_0>0, \end{cases}\\
			 \Gamma_{xy}^{t}&:=\Gamma_{(0,x-y)},\\
			 \Sigma_{x}^{ts}&:=\Gamma_{(t-s,0)}
		\end{split}
\end{align}
for $x,y\in\R^d, t,s\in \R$ and $k=(k_0,\overline k)\in \N_0\times \N_0^d$, where $\Gamma_h$ is defined as in Definition \ref{Def:PolynomialRegularityStructure}. 
\end{definition}

\begin{remark}\label{Rem:PolynomialModel}	
\begin{itemize}
\item[a)] We defined $\Pi_x^t$ here as a function and not as a distribution, this has to be interpreted as the regular distribution generated by the function.
\item[b)] If we define the model $\tilde M$ as in \ref{Rem:Model} b) then it acts on the polynomials in the usual way.
\end{itemize}
\end{remark}

For functions taking values in $\mathcal{T}_{<\gamma}$ we have a norm (weighted in time) and we can define the function space
of modelled distributions using this norm.

\begin{definition}
Let $\mathscr{T}$ be a regularity structure equipped with an inhomogeneous model $M=(\Pi, \Gamma,\Sigma)$ over $\R\times\R^d$ and $\gamma>0$, $\eta\in\R$.\\
For $T>0$ and ${\color{blue}F}:(0,T]\times\R^d\to \mathcal{T}_{<\gamma}$ we define 
\begin{align}\label{ModelledDistributionEstimate}
\begin{split}
\|{\color{blue}F}\|_{\infty;\gamma,\eta;T}
&:=\sup_{t\in (0,T]} \sup_{x\in\R^d} \sup_{m\in A,m<\gamma} |t|_0^{(m-\eta) \lor 0} \|{\color{blue}F}(t,x)\|_m,\\
[{\color{blue}F}]_{\gamma,\eta;T}^{\text{space}}
& := \sup_{t\in (0,T]} \sup_{x,y\in\R^d, |x-y|\leq 1} \sup_{m\in A,m<\gamma} \frac{\|{\color{blue}F}(t,x)-\Gamma^{t}_{xy}{\color{blue}F}(t,y)\|_m}{|t|_0^{\eta-\gamma}|x-y|^{\gamma-m}}\\
[{\color{blue}F}]_{\gamma,\eta;T}^{\text{time}}
&:=\sup_{t,s\in (0,T],|t-s|\leq |t,s|_0^{s_0}} \sup_{x\in\R^d} \sup_{m\in A,m<\gamma} \frac{\|{\color{blue}F}(t,x)-\Sigma^{ts}_{x}{\color{blue}F}(s,x)\|_m}{|t,s|_0^{\eta-\gamma}|t-s|^{(\gamma-m)/s_0}},\\
\|{\color{blue}F}\|_{\gamma,\eta;T}
&:= \|{\color{blue}F}\|_{\infty;\gamma,\eta;T} + [{\color{blue}F}]^{\text{space}}_{\gamma,\eta;T}\\
\triple {\color{blue}F}\triple _{\gamma,\eta;T}
&:= \|{\color{blue}F}\|_{\gamma,\eta;T} + [{\color{blue}F}]^{\text{time}}_{\gamma,\eta;T}
\end{split}
\end{align}
and the space of \textit{inhomogeneous modelled distributions}
\begin{align*}
\mathscr{D}^{\gamma,\eta}(T;M):=\{{\color{blue}F}:(0,T]\times\R^d\to \mathcal{T}_{\gamma}| \triple {\color{blue}F}\triple _{\gamma,\eta;T}<\infty \}.
\end{align*}
If $\overline M=(\overline\Pi, \overline\Gamma,\overline\Sigma)$ is a second model and $\mathscr{D}^{\gamma,\eta}(T;\overline{M})$ the corresponding space of modelled distribution, we define the distance between functions ${\color{blue}F}\in  \mathscr{D}^{\gamma,\eta}(T;M)$ and $\overline{{\color{blue}F}}\in \mathscr{D}^{\gamma,\eta}(T;\overline{M})$ as
\begin{align*}
&\|{\color{blue}F};\overline{{\color{blue}F}}\|_{\gamma,\eta;T}:=  \|{\color{blue}F}-\overline{{\color{blue}F}}\|_{\infty;\gamma,\eta;T}\\
& +\sup_{t\in (0,T]} \sup_{x,y\in\R^d, |x-y|\leq 1} \sup_{m\in A,m<\gamma} \frac{\|{\color{blue}F}(t,x)-\Gamma^{t}_{xy}{\color{blue}F}(t,y)-(\overline{{\color{blue}F}}(t,x)-\overline{\Gamma}^{t}_{xy}\overline{{\color{blue}F}}(t,y))\|_m}{|t|_0^{\eta-\gamma}|x-y|^{\gamma-m}},\\
&\triple{\color{blue}F};\overline{{\color{blue}F}}\triple_{\gamma,\eta;T}:=  \|{\color{blue}F}-\overline{{\color{blue}F}}\|_{\infty;\gamma,\eta;T} \\
& +\sup_{t,s\in (0,T]} \sup_{x\in\R^d} \sup_{m\in A,m<\gamma} \frac{\|{\color{blue}F}(t,x)-\Gamma^{t}_{xy}{\color{blue}F}(t,x)-(\overline{{\color{blue}F}}(t,x)-\overline{\Sigma}^{ts}_{x}\overline{{\color{blue}F}}(s,x))\|_m}{|t,s|_0^{\eta-\gamma}|t-s|^{(\gamma-m)/s_0}}.
\end{align*}
\end{definition}

\begin{remark}\label{Rem:ModelledDistributions}
\begin{itemize}
\item[a)] We will often write $F_t:=F(t,\cdot)$.
\item[b)] Let us emphasis that the space $\mathscr{D}^\gamma(T;M)$ depends on $\Gamma$ and $\Sigma$, but not on $\Pi$. We will drop the dependence on the model from time to time.
\item[c)] Let $M=(\Pi, \Gamma,\Sigma)$  be an inhomogeneous model and $\tilde M=(\tilde \Pi, \tilde \Gamma)$ be the model defined in Remark \ref{Rem:Model}. Then the corresponding spaces of modelled distributions $\tilde D^{\gamma,\eta}(T;\tilde M)$ defined in \cite{Hairer_2014} satisfy $\tilde D^{\gamma,\eta}(T;\tilde M)\subset D^{\gamma,\eta}(T; M)$ (in general $\subsetneq$).
\end{itemize}
\end{remark}

One of the most important results is the existence of the reconstruction operator, which is a mapping from $(0,T]\times\mathscr{D}^{\gamma,\eta}(T;M)$ onto the space of H\"older distributions $C^{\alpha}(\R^d)$. In case of an inhomogeneous model with time regularity $\delta>0$ this yields functions in a weighted space $C^{\overline\delta,\alpha}_{\eta-\gamma;T}$.

\begin{theorem}[Reconstruction]\label{ReconstructionTheorem}
Let $M=(\Pi,\Gamma, \Sigma)$ be an inhomogeneous model for a regularity structure $\mathscr{T}$ on $\R\times\R^d$, $\alpha:=0 \land \min A$ and $\gamma >0$, $\eta \in \R$.
\begin{itemize}
\item[(i)]  There exists a unique linear map
\begin{align*}
\mathcal{R}:(0,T]\times \mathscr{D}^{\gamma,\eta}(T;M)\rightarrow C^{\alpha}(\R^d)
\end{align*}
such that
\begin{align}\label{ReconstructionTheoremEstimate}
\abs{\left(\mathcal{R}_t{\color{blue}F}_t-\Pi_x^t{\color{blue}F}_t(x)\right)(\varphi_x^\lambda)}\leq C\lambda^\gamma |t|_0^{\eta-\gamma} \|{\color{blue}F}\|_{\gamma,\eta;T} \|\Pi\|_{\gamma;T},
\end{align}
uniformly over all $\varphi\in\mathcal{B}_r$, $\lambda\in (0,1]$, ${\color{blue}F}\in\mathscr{D}^{\gamma,\eta}(T;M)$, $t\in(0,T]$ and $x\in\R^d$, where $r$ is the smallest integer such that $r>|\alpha|$.
\item[(ii)] If $\Pi$ has additionally time regularity $\delta>0$ then for any $\overline \delta\in (0,\delta]$ such that $\overline{\delta}<m-\zeta$ for all $\zeta,m\in A \cup \{\gamma\}$ with $\zeta<m\leq \gamma$ the function $t\mapsto \mathcal{R}_t{\color{blue}F}_t$ satisfies
\begin{align*}
\|\mathcal{R}_\cdot{\color{blue}F}_\cdot\|_{C^{\overline\delta,\alpha}_{\eta-\gamma;T}}\leq C \|\Pi\|_{\delta,\gamma;T} (1+\|\Sigma\|_{\gamma;T})\triple {\color{blue}F}\triple _{\gamma,\eta;T}.
\end{align*}
\item[(iii)] Let $\overline M=(\overline\Pi,\overline\Gamma,\overline\Sigma)$ be another model for $\mathscr{T}$ and $\overline{\mathcal{R}}$ be the corresponding operator as in (i). If $\Pi$ and $\overline{\Pi}$ have time regularity $\delta >0$  and $\overline{\delta}$ is as in (ii), for every ${\color{blue}F}\in \mathscr{D}^{\gamma,\eta}(T;M)$ and $\overline{{\color{blue}F}}\in \mathscr{D}^{\gamma,\eta}(T;\overline{M})$ the maps $t\mapsto \mathcal{R}_t{\color{blue}F}_t$ and $t\mapsto \overline{\mathcal{R}}_t\overline{{\color{blue}F}}_t$ satisfy
\begin{align*}
\|\mathcal{R}_\cdot{\color{blue}F}_\cdot-\mathcal{R}_\cdot\overline{{\color{blue}F}}_\cdot\|_{C^{\overline\delta,\alpha}_{\eta-\gamma;T}}\leq C \|\Pi;\overline \Pi\|_{\delta,\gamma;T} \triple {\color{blue}F};\overline{{\color{blue}F}}\triple _{\gamma,\eta;T},
\end{align*}
where $C$ depends on the maximum of $\|\Pi\|_{\delta,\gamma;T},\|\overline \Pi\|_{\delta,\gamma;T}, \triple {\color{blue}F}\triple _{\gamma,\eta;T}$ and \\ $\triple \overline{{\color{blue}F}}\triple _{\gamma,\eta;T}$.
\end{itemize}
\end{theorem}

\begin{remark}\label{ReconstructionTheoremRemark}
\begin{itemize}
\item[a)] We call $\mathcal{R}$ the reconstruction operator.
\item[b)] The map $\mathcal{R}_t$ can be viewed as a nonlinear map from the space of all inhomogeneous models and all modelled distributions into $C^{\alpha}$. For $\Gamma$ and $\Sigma$ fixed, the map $\mathcal{R}$ is actually bilinear in ${\color{blue}F}$ and $\Pi$.
\item[c)] If $\Pi_x^t{\color{blue}\tau}$ is a continuous function for every ${\color{blue}\tau}\in \mathcal{T}$ and every $(t,x)\in\R^{1+d}$, then $\mathcal{R}_\cdot{\color{blue}F}_{\cdot}$ is also a continuous function and $\mathcal{R}$ is given by the explicit formula
\begin{align*}
\left(\mathcal{R}_t{\color{blue}F_t}\right)(x)=\left(\Pi^t_x{\color{blue}F_t}(x)\right)(x).
\end{align*}
\item[d)] If $M$ is an inhomogeneous model with $\mathscr{D}^{\gamma,\eta}(T;M)$ the space of inhomogeneous distributions and $\mathcal{R}_{\cdot}$ the reconstruction operator and $\tilde{M}$, $\tilde{\mathscr{D}}^{\gamma,\eta}(T;\tilde{M})$ as in Remark \ref{Rem:ModelledDistributions}, then we have for the reconstruction operator $\tilde{\mathcal{R}}$ from \cite{Hairer_2014} that $\tilde{\mathcal{R}}=\mathcal{R}_{\cdot}$ on $\tilde{\mathscr{D}}^{\gamma,\eta}(T;\tilde{M})$ (note, that the domain of definition of $\tilde{\mathcal{R}}$ is in general smaller than the one of $\mathcal{R}_{\cdot}$).
\end{itemize}
\end{remark}

Next, we lift different analytic operations to the level of regularity structures. We cannot define all operations on the whole structure and to take that into account we first introduce sectors of a regularity structure.

\begin{definition}
Given a regularity structure $\mathscr{T}=(A,\mathcal{T},G)$ equipped with a model $M=(\Pi,\Gamma,\Sigma)$ and $\gamma >0$, $\eta\in\R$. We call a subspace $V\subset \mathcal{T}$ a \textit{sector} if
\begin{itemize}
\item[(i)] $V$ is invariant under the action of $G$, i.e. $\Gamma V\subset V$ for every $\Gamma\in G$,
\item[(ii)] $V$ can be written as $V=\bigoplus_{\alpha\in A}V_\alpha$ with $V_\alpha\subset \mathcal{T}_\alpha$.
\end{itemize}
Furthermore, we define the space of modelled distributions whose range is contained in a sector $V$ as
\begin{align*}
\mathscr{D}^{\gamma,\eta}(T;V;M)\coloneqq \{ {\color{blue}F}\in\mathscr{D}^{\gamma,\eta}(T;M):{\color{blue}F}(t,x)\in V\text{ for all }(t,x)\in(0,T]\times\R^d\}.
\end{align*}
We call the minimal $\alpha_0<\gamma$ such that $V_{\alpha_0}\neq 0$ the regularity of the sector. We will write  
\begin{align*}
\mathscr{D}^{\gamma,\eta}_{\alpha_0}(T;V;M)\coloneqq\left\{{\color{blue}f}\in\mathscr{D}^{\gamma,\eta}(T;V;M):{\color{blue}f}(x)\in \mathcal{T}_{\geq\alpha_0}=\bigoplus_{\beta\geq\alpha_0}\mathcal{T}_\beta\right\}
\end{align*}
to indicate the regularity of the sector in the notation of the space of modelled distributions.
\end{definition}

\begin{definition}
Given a regularity structure $\mathscr{T}=(A,\mathcal{T},G)$ and two sectors $V,\bar V\subset \mathcal{T}$, a \textit{product} on $(V,\bar V)$ is a bilinear continuous map $\star:V\times\bar V\rightarrow \mathcal{T}$ such that, for any ${\color{blue}\tau}\in V_\alpha$, ${\color{blue}\bar\tau}\in\bar V_\beta$ and $\Gamma\in G$, one has
\begin{align*}
{\color{blue}\tau}\star{\color{blue}\bar\tau}\in \mathcal{T}_{\alpha+\beta}\qquad\text{and}\qquad\Gamma({\color{blue}\tau}\star{\color{blue}\bar\tau})=\Gamma{\color{blue}\tau}\star\Gamma{\color{blue}\bar\tau},
\end{align*}
\end{definition}

With such a product at hand, we want to know if the pointwise product of two modelled distributions yields again a modelled distribution. The next theorem answers this question. 

\begin{theorem}\label{TheoremAbstractProduct}
Given a regularity structure $\mathscr{T}=(A,\mathcal{T},G)$ and sectors $V,\bar V\subset \mathcal{T}$. Let ${\color{blue}f_1}\in\mathscr{D}^{\gamma_1,\eta_1}_{\alpha_1}(T;V;M)$, ${\color{blue}f_2}\in\mathscr{D}^{\gamma_2,\eta_2}_{\alpha_2}(T;\bar V;M)$, and let $\star$ be a product on $(V,\bar V)$. Then, the function ${\color{blue}f}$ given by ${\color{blue}f}(x)={\color{blue}f_1}(x)\star{\color{blue}f_2}(x)$ belongs to $\mathscr{D}^{\gamma,\eta}_\alpha(T;M)$ with
\begin{align*}
\alpha=\alpha_1+\alpha_2,\quad \gamma=(\gamma_1+\alpha_2)\land (\gamma_2+\alpha_1), \quad \eta=(\eta_1+\alpha_2)\land (\eta_2+\alpha_1)\land (\eta_1+\eta_2).
\end{align*}
\end{theorem}

\subsection{Integration against Singular Kernels}

Next we want to include the integration against singular kernels into the setting of regularity structures. For smoothing kernels (in space and time) this has been done already and we only cite the needed results, but for non-smoothing kernels like the Riesz transform the result is new. Additionally, the Riesz transform acts only in the spacial variable which is the reason to consider inhomogeneous models.\\
To lift such nonlocal operators we need regularity structures which contain the polynomials in a certain sense.

\begin{definition}
We call a regularity structure $\mathscr{T}=(A,\mathcal{T},G)$ normal, if there exists a sector $\overline{\mathcal{T}}\subset \mathcal{T}$ 
isomorphic to the space of abstract polynomials in $1+d$ commuting variables, 
meaning $\overline{\mathcal{T}}_{\alpha}\neq 0$ if and only if $\alpha\in\N$, and one can find basis 
vectors ${\color{blue}X^k}$ of $\overline{\mathcal{T}}_{\abs k}$ such that every element $\Gamma\in G$ 
acts on $\overline{\mathcal{T}}$ by $\Gamma{\color{blue}X^k}=({\color{blue}X}+h{\color{blue}1})^k$ for 
some $h\in\R^d$.
\end{definition}

\begin{definition}
Let $\mathscr{T}=(A,\mathcal{T},G)$ be a normal regularity structure. Given a sector $V\subset \mathcal{T}$, a linear map $\mathcal{I}:V\rightarrow \mathcal{T}$ is an \textit{abstract integration map of order} $\beta \geq 0$ if it satisfies the following properties:
\begin{itemize}
\item[(i)] For every $\alpha\in A$ one has $\mathcal{I}:V_\alpha\rightarrow \mathcal{T}_{\alpha+\beta}$.
\item[(ii)] For every ${\color{blue}\tau}\in V\cap\overline{\mathcal{T}}$ one has $\mathcal{I}{\color{blue}\tau}=0$.
\item[(iii)] For every ${\color{blue}\tau}\in V$ and every $\Gamma\in G$ one has $\mathcal{I}\Gamma{\color{blue}\tau}-\Gamma\mathcal{I}{\color{blue}\tau}\in\overline{\mathcal{T}}$.
\end{itemize}
In the case of $\alpha+\beta\notin A$, (i) is understood as $\mathcal{I}{\color{blue}\tau}=0$ if ${\color{blue}\tau}\in V_\alpha$.
\end{definition}

\subsubsection{Integration against Regularizing Kernels}
Let us first turn to regularising kernels and make precise what we mean exactly when we said that $K$ ``improves regularity by order $\beta>0$''.

\begin{definition}\label{AssumptionK}
We call a kernel $K:\R^{1+d}\setminus\{0\}\rightarrow\R$ \textit{$\beta$-regularising} if it is smooth and if one can write $K=\sum_{n\geq 0}K_n$ where each of the kernels $K_n:\R\times\R^d\rightarrow\R$ is smooth and compactly supported in a ball of radius $2^{-n}$ around the origin.\\
Furthermore, we assume that for every multi-index $k\in\N_0^{1+d}$, there exists a constant $C$ such that
\begin{align*}
\underset{(t,x)\in \R\times\R^d}{\sup}\abs{D^kK_n(t,x)}\leq C 2^{(d+s_0-\beta+\abs k)n}
\end{align*}
holds uniformly in $n$. Finally, we assume that there exists $N\in\N$ sufficiently large such that $\int_\R \int_{\R^d} K_n(t,x)P(t,x)dx \; dt=0$ for every polynomial $P$ of degree at most $N$.
\end{definition}

\begin{remark}\label{RemarkDecompKernel}
Lemma 5.5 in \cite{Hairer_2014} shows, that it is possible under certain scaling behaviours to decompose a smooth function $\bar K:\R\times \R^d\setminus\{0\}\rightarrow\R$ as $\bar K(t,x)=K(t,x)+R(t,x)$ such that $R$ is $C^\infty$ and $K$ is regularising in the sense of Definition \ref{AssumptionK}.
\end{remark}

From now on, we assume that $K(t,x)=0$ for $t<0$, i.e, $K$ is non-anticipative. 

To define models that are  ``compatible'' with a regularizing kernel we need a function that helps us to force $\Pi_x^t\mathcal{I}{\color{blue}\tau}$ to vanish at the correct order. For this, given a kernel $K$ satisfying Assumption \ref{AssumptionK}, for every $t\in\R,x\in\R^d$ and ${\color{blue}\tau}\in \mathcal{T}_\alpha$ with $\alpha\in A$, we define a function $\mathcal{J}(t,x):\mathcal{T}\rightarrow\overline{\mathcal{T}}$ by
\begin{align}\label{DefinitionJ}
\mathcal{J}(t,x){\color{blue}\tau}\coloneqq \sum_{\abs k<\alpha+\beta}\sum_{n\geq 0}\frac{{\color{blue}X^k}}{k!}\int_0^t\langle\Pi_x^s \Sigma_x^{st}{\color{blue}\tau},D^kK_n(t-s,x-\cdot)\rangle \d s,
\end{align}
which reads formally as
\begin{align*}
\mathcal{J}(t,x){\color{blue}\tau}
 & =\sum_{\abs k<\alpha+\beta}\frac{{\color{blue}X^k}}{k!}\int_0^t\int_{\R^d}(\Pi_x^s\Sigma^{st}_{x}{\color{blue}\tau})(s,y)D^kK(t-s,x-y)dy \; ds\\
 & = \sum_{\abs k<\alpha+\beta}\frac{{\color{blue}X^k}}{k!}((\Pi_x^\cdot\Sigma^{\cdot t}{\color{blue}\tau})\ast D^k K)(t,x).
\end{align*}
$\mathcal{J}(t,x){\color{blue}\tau}$ is well-defined since the sum converges absolutely by \cite[Lemma 5.19]{Hairer_2014}.
We can now define suitable models.

\begin{definition}\label{DefinitionAdmissible}
Let $K$ be a $\beta$-regularising kernel and $\mathscr{T}$ a normal regularity structure. Further, let $\mathcal{I}:V\rightarrow \mathcal{T}$ be an abstract integration map of order $\beta> 0$ defined on some sector $V$. We say that an inhomogeneous model $M=(\Pi,\Gamma, \Sigma)$ \textit{realizes} $K$ for $\mathcal{I}$ if the identity
\begin{align}\label{RealizingModelIdentities}
\Pi_x^t\mathcal{I}{\color{blue}\tau}=K\ast\Pi_x^\cdot\Sigma^{\cdot t} {\color{blue}\tau}-\Pi_x^t\mathcal{J}(t,x){\color{blue}\tau}
\end{align}
holds for every ${\color{blue}\tau}\in V$ with $\abs{\color{blue}\tau}\leq N$ and if for all $s,t\in\R, x,y \in\R^d$
\begin{align*}
\Gamma^{t}_{xy}(\mathcal{I}+\mathcal{J}(t,y))=(\mathcal{I}+\mathcal{J}(t,x))\Gamma^{t}_{xy},\\
\Sigma^{ts}_{x}(\mathcal{I}+\mathcal{J}(s,x))=(\mathcal{I}+\mathcal{J}(t,x))\Sigma^{ts}_{x}.
\end{align*}
\end{definition}

The existence of an abstract integration map and a realizing model is guaranteed by the following theorem under rather mild assumptions.

\begin{theorem}[Extension Theorem]\label{ExtensionTheorem}
Let $\mathscr{T}=(A,\mathcal{T},G)$ be a normal regularity structure, let $\beta>0$ and $V\subset \mathcal{T}_{\leq\bar\gamma}$, $\bar\gamma>0$, be a sector with the property that for every $\alpha\notin\N$ with $V_\alpha\neq\{0\}$, one has $\alpha+\beta\notin\N$. Let furthermore $W\subset V$ be a subsector of $V$ and let $K$ be a $\beta$-smoothing kernel. Let $M=(\Pi,\Gamma, \Sigma)$ be an inhomogeneous model for $\mathscr{T}$, and let $\mathcal{I}:W\rightarrow \mathcal{T}$ be an abstract integration map of order $\beta$ such that $M$ realizes $K$ for $\mathcal{I}$.\\
Then, there exists a regularity structure $\hat{\mathscr{T}}$ containing $\mathscr{T}$, in the sense that $\mathcal{T}\subset\mathcal{\hat T}$ and the action of $\hat G$ on $\mathcal{T}$ is compatible with that of $G$, a model $\hat M=(\hat\Pi,\hat{\Gamma},\hat{\Sigma})$ for $\hat{\mathscr{T}}$ extending $M$, and an abstract integration map $\hat{\mathcal{I}}:V\rightarrow \mathcal{T}$ of order $\beta$ such that the model $\hat M$ realizes $K$ for $\hat{\mathcal{I}}$, and the map $\hat{\mathcal{I}}$ extends $\mathcal{I}$ from $W$ to $V$.
\end{theorem}

\begin{proof}
The proof is virtually the same as for the statement regarding homogeneous models, see \cite[Theorem 5.14]{Hairer_2014}. The only differences are in the definition of the extended model $(\hat\Pi,\hat\Gamma,\hat\Sigma)$ and the additional algebraic properties that the inhomogeneous model requires.\\
Using the same notation as in the original proof, we define
\begin{align*}
\hat{\Pi}_x^t(a,b) \coloneqq \Pi_x^ta + K\ast\Pi_x^{\cdot}\Sigma_x^{\cdot t}b - \Pi_x^t\mathcal{J}(t,x)b,
\end{align*}
where $(a,b)\in T\oplus\bar W$. It is then straightforward to verify that $\hat\Pi$ satisfies (\ref{RealizingModelIdentities}) on the extended model space $\mathcal{T}\oplus\bar W$.\\
In complete analogy to the original proof, we define
\begin{align*}
\hat\Gamma_{xy}^t & \coloneqq (\Gamma_{xy}^t,\mathcal J(t,x)\Gamma_{xy}^t-\Gamma_{xy}^t\mathcal J(t,y)),\\
\hat\Sigma_x^{st} & \coloneqq (\Sigma_x^{st},\mathcal J(s,x)\Sigma_x^{st}-\Sigma_x^{st}\mathcal J(t,x))
\end{align*}
and obtain that $(\hat\Pi,\hat\Gamma,\hat\Sigma)$ is a model realising $K$ for $\hat{\mathcal I}$.
\end{proof}

\begin{remark}
The above theorem allows the case for $W$ to be empty, which implies, under the above assumptions, the existence of abstract integration maps as well as realizing models.
\end{remark}

Next, we aim at the actual lift of the convolution with $K$ to the level of regularity structures. We define

\begin{align}\label{DefinitionN}
(\mathcal{N}{\color{blue}F})(t,x)\coloneqq\sum_{\abs k<\gamma+\beta}\sum_{n\geq 0}\frac{{\color{blue}X^k}}{k!}\int_0^t\langle\mathcal{R}_s{\color{blue}F}_s-\Pi_x^s\Sigma^{st}_{x}{\color{blue}F}_t(x),D^kK_n(t-s,x-\cdot)\rangle ds.
\end{align}

Then we have the following result, see \cite[Theorem 2.21]{hairer2017discretisations} for the proof.

\begin{theorem}\label{TheoremAbstractIntegration}
Given a normal regularity structure $\mathscr{T}=(A,\mathcal{T},G)$ equipped with an inhomogeneous model $M=(\Pi,\Gamma,\Sigma)$ and a sector $V\subset \mathcal{T}$ of regularity $\alpha_0$. Let $\alpha:=0\land \alpha_0$, $\gamma,\beta>0$ with  $\gamma+\beta\notin\N$  and $K$ be a $\beta$-regularizing kernel with $N\geq \gamma+\beta \lor |\alpha|$. Further, let $\mathcal{I}:V\rightarrow \mathcal{T}$ be an abstract integration map of order $\beta$ and let $M$ realize $K$ for $\mathcal{I}$. For $\eta\in (\gamma-s_0,\gamma)$ and $T>0$ define the operator $\mathcal{K}$ for  ${\color{blue}F}\in\mathscr{D}^{\gamma,\eta}(T;V;M)$ by
\begin{align}\label{eq:DefAbstractK}
(\mathcal{K}{\color{blue}F})(t,x)=\mathcal{I}{\color{blue}F}(t,x)+\mathcal{J}(t,x){\color{blue}F}(t,x)+(\mathcal{N}{\color{blue}F})(t,x).
\end{align}
Then, $\mathcal{K}$ maps $\mathscr{D}^{\gamma,\eta}(T;V;M)$ onto $\mathscr{D}^{\gamma+\beta,(\eta\land \alpha)+\beta}(T;M)$ with
\begin{align*}
\triple \mathcal{K}{\color{blue}F}\triple _{\gamma+\beta,(\eta\land \alpha)+\beta;T} \leq C \triple {\color{blue}F}\triple _{\gamma,\eta;T} \|\Pi\|_{\gamma;T} \|\Sigma\|_{\gamma;T}(1+\|\Gamma\|_{\gamma+\beta;T}+\|\Sigma\|_{\gamma+\beta;T})
\end{align*}
and the identity
\begin{align*}
\mathcal{R}_t\mathcal{K}{\color{blue}F}_t=(K\ast\mathcal{R}_{\cdot}{\color{blue}F}_{\cdot})(t)
\end{align*}
holds for $t\in (0,T]$.\\
If $\overline M=(\overline\Pi,\overline\Gamma,\overline\Sigma)$ is another inhomogeneous model realizing $K$ for $\mathcal{I}$ under the same assumptions as above and $\overline{\mathcal{K}}$ be defined as in \eqref{eq:DefAbstractK}, then we have for all ${\color{blue}F}\in\mathscr{D}^{\gamma,\eta}(T;V;M)$ and $\overline{{\color{blue}F}}\in\mathscr{D}^{\gamma,\eta}(T;V;\overline{M})$
\begin{align*}
\triple \mathcal{K}{\color{blue}F};\overline{\mathcal{K}}\overline{{\color{blue}F}}\triple _{\gamma+\beta,(\eta\land \alpha)+\beta;T} \leq C (\triple {\color{blue}F};\overline{{\color{blue}F}}\triple _{\gamma,\eta;T} + \|M;\overline{M}\|_{\gamma+\beta;T}),
\end{align*}
where $C$ depends on the maximum of $\triple {\color{blue}F}\triple _{\gamma,\eta;T},\triple \overline{{\color{blue}F}}\triple _{\gamma,\eta;T},\|\Pi\|_{\gamma;T},\|\overline{\Pi}\|_{\gamma;T},\|\Gamma\|_{\gamma+\beta;T}$, $\|\Sigma\|_{\gamma+\beta;T},\|\overline{\Gamma}\|_{\gamma+\beta;T}$ and $\|\overline{\Sigma}\|_{\gamma+\beta;T}$.
\end{theorem}

\subsubsection{Integration against Non-regularizing Kernels}
Now, let us turn to a non-regularising kernel $R$ depending only on the spatial variables, 
which corresponds to the case $\beta=0$ 
in the above definitions. The convolution with such a kernel is in general not a well-defined 
operator and hence we cannot expect that only the behaviour of the singularity is enough to 
lift it without any additional assumptions.\\
Our framework is, that we can write $R= L G$ for some kernel $G$ on $\R^d\setminus\{0\}$ which 
is smoothing of order $\beta>0$ and $L$ a differential operator with constant coefficients of 
order less or equal than $\beta$ acting in the spatial coordinate directions.  This is slightly 
more general than assuming that $L$ is exactly of order $\beta$ and even so we do not need this 
generalization for the Riesz transform it can be helpful in other applications.\\
Since such a kernel does not improve the regularity in time we can only aim to prove that we get a 
mapping from $\mathcal{D}^{\gamma}$ into itself, which is the reason to call all these kernels 
non-regularizing. Let us make the definition precise.

\begin{definition}\label{AssumptionR}
We call a kernel $R:\R^d\rightarrow\R$ non-regularising if there exist
$\beta_L\in\N_0$, $0\neq\beta\geq \beta_L$, a kernel $G:\R^d\setminus\{0\} \to\R$ 
and a differential operator $L$ of order $\beta_L$ with constant 
coefficients such that $R=LG$.\\
Furthermore, we assume $G=\sum_{n\geq 0}G_n$ where each of the kernels $G_n:\R^d\rightarrow\R$ is 
smooth and compactly supported in a ball of radius 
$2^{-n}$ around the origin and that for every multi-index $k\in\N_0^{d}$, there exists a constant $C$ such that
\begin{align*}
\underset{x\in \R^d}{\sup}\abs{D^kG_n(x)}\leq C 2^{(d+\abs k-\beta)n}
\end{align*}
holds uniformly in $n$. Finally, we assume that there exists $N\in\N$ sufficiently large such that $\int_{\R^d} G_n(x)P(x)dx=0$ 
for every polynomial $P$ of degree at most $N$.
\end{definition}
We have for such a kernel $R=\sum_{n\geq 0} R_n:=\sum_{n\geq 0}LG_n$, 
$\underset{x\in \R^d}{\sup}\abs{D^kR_n(x)}\leq C 2^{(d+\abs k+\beta_L-\beta)n}$ and
$\int_{\R^d} R_n(x)P(x)dx=0$ 
for every polynomial $P$ of degree at most $N+\beta_L$ by definition.

\begin{example}\label{example}
\begin{itemize}
\item[a)] The Riesz transform $R_i$ given by $\partial_i \Delta^{-1/2}$ fits in our setting, since the Kernel corresponding to $\Delta^{-1/2}$ is $1$-regularizing on $\R^d$.
\item[b)] The Helmholtz (or Leray) projection is given by $P=1-\nabla \Delta^{-1} \div$, 
so $Pu=u-\nabla K \ast_{\R^d} \div u$ where $K$ is the kernel corresponding to $\Delta^{-1}$. 
Integration by parts yields $K \ast_{\R^d} \div u= - (\nabla K) \ast_{\R^d}  u$ and we get 
$Pu=u+\nabla (\sum_{i=1}^d \partial_i K \ast_{\R^d} u_i)=u+\sum_{i=1}^d \nabla \partial_i K \ast_{\R^d} u_i$.
 The kernels $\nabla \partial_i K$ are non-regularizing in the above sense and so the results of this section 
 can be applied to $P-1$, which also yields a way to lift $P$ itself.
\item[c)] The convolutions with the kernel $K(x)=|x|^{-d}$ is not a well-defined operator. Considering any 
positive smooth function which is constant and non-zero in a small ball around the origin, the convolution is infinite in this ball. 
\end{itemize}
\end{example}

For a non-regularizing kernel $R$, $t\in\R,x\in\R^d$ and ${\color{blue}\tau}\in \mathcal{T}_\alpha$ with $\alpha\in A$, $\alpha+\beta\notin \N_0$, we formally define the function $\mathcal{J}(t,x):\mathcal{T}\rightarrow\overline{\mathcal{T}}$ by
\begin{align}\label{DefinitionJR}
\mathcal{J}(t,x){\color{blue}\tau}\coloneqq \sum_{\abs k<\alpha}\sum_{n\geq 0}\frac{{\color{blue}X^k}}{k!} \langle\Pi_x^t{\color{blue}\tau},D^kR_n(x-\cdot)\rangle.
\end{align}
Note, that we will need the additional structure of $R$ to guarantee that the sum converges absolutely and that
$\mathcal{J}(t,x){\color{blue}\tau}$ is well-defined. We are interested in models that realize $R$ in the following sense.

\begin{definition}\label{DefinitionAdmissibleR}
Let $R$ be a non-regularising kernel and $\mathscr{T}$ a normal regularity structure. Further, let $\mathcal{I}:V\rightarrow \mathcal{T}$ be an abstract integration map of order $0$. We say that an inhomogeneous model $M=(\Pi,\Gamma, \Sigma)$ \textit{realizes} $R$ for $\mathcal{I}$ if the identity
\begin{align}\label{RealizingModelIdentitiesR}
\Pi_x^t\mathcal{I}{\color{blue}\tau}=K\ast_{\R^d}\Pi_x^t {\color{blue}\tau}-\Pi_x^t\mathcal{J}(t,x){\color{blue}\tau}
\end{align}
holds for every ${\color{blue}\tau}\in \mathcal{T}$ with $\abs{\color{blue}\tau}\leq N$ and if for all $t,s\in\R, x,y \in\R^d$
\begin{align*}
\Gamma^{t}_{xy}(\mathcal{I}+\mathcal{J}(t,y))=(\mathcal{I}+\mathcal{J}(t,x))\Gamma^{t}_{xy},\\
\Sigma^{ts}_{x}(\mathcal{I}+\mathcal{J}(s,x))=(\mathcal{I}+\mathcal{J}(t,x))\Sigma^{ts}_{x}.
\end{align*}
\end{definition}

To show that $\mathcal{J}$ is well-defined and that the right-hand side of \eqref{RealizingModelIdentitiesR} 
makes sense, we prove an estimate for $\langle\Pi_x^t{\color{blue}\tau},D^kR_n(x-\cdot)\rangle$ first.

\begin{lemma}\label{lemma:firstestimate}
Let $R$ be a non-regularising kernel and $\mathscr T$ a regularity structure equipped with an inhomogeneous model $(\Pi,\Gamma,\Sigma)$. Then 
for every $\alpha\in A$ with $\alpha+\beta\notin\N_0$, $T>0$, ${\color{blue}\tau}\in  \mathcal{T}_{\alpha}$, $k=(k_1,\dots,k_d) \in\N_0^d$ we have
\begin{align*}
\abs{\langle\Pi^t_x{\color{blue}\tau},D^kR_n(x-\cdot)\rangle} \leq C 2^{(\beta_L-\beta-\alpha +\abs k)n}\norm{\Pi}_{\alpha;T}
\end{align*}
uniformly over $t\in[0,T]$ and $x \in\R^d$
\end{lemma}

For $|k|\leq \alpha$ this implies that $\mathcal J$ is well-defined. Also the proof of the extension theorem makes 
use of this bound. In detail, this is required to show the analytic 
bound on $\hat\Gamma^t_{xy}$ as can be seen in the proof of Lemma 5.21 in \cite{Hairer_2014}.

We begin the proof of Lemma \ref{lemma:firstestimate} with an estimate which is useful in regimes 
of large scales, meaning for those $n$ such that $\norm{x-y}_{\mathfrak s}\leq 2^{-n}$. For this case we can follow the proof of Lemma 5.18 in \cite{Hairer_2014}.
To shorten the notation let us define
\begin{align*}
R^{k,\alpha}_{n,xy}(z) \coloneqq D^kR_n(y-z)-\sum_{\abs{k+l}_{\mathfrak s}<\alpha}\frac{(y-x)^l}{l!}D^{k+l}R_n(x-z).
\end{align*}

\begin{lemma}
Let $n\in\N$. Under the conditions of Lemma \ref{lemma:firstestimate} we have for $\tau\in  \mathcal{T}_{\alpha}$
\begin{align*}
\abs{(\Pi^t_y{\color{blue}\tau})(R^{k,\alpha}_{n,xy})}\leq C \norm{\Pi}_{\alpha;T}(1+\norm{\Gamma}_{\alpha;T})\sum_{\delta>0}2^{\delta n}\norm{x-y}^{\delta+\alpha+\beta-\beta_L-\abs k}
\end{align*}
uniformly over $t\in[0,T]$, $x,y\in \R^d$ and $\norm{x-y}_{\mathfrak s}\leq 2^{-n}$ and where the sum runs over finitely many $\delta$.
\end{lemma}

\begin{remark}
Note, that $\beta-\beta_L+(\gamma_{A_\gamma}\wedge\mathfrak s_{A_\gamma})> 0$ \\
where
\begin{align*}
\gamma_{A_\gamma} & =\min\left\{\alpha-\min\{\alpha'\in A: \alpha'<\alpha\}:\alpha\in A,\alpha<\gamma\right\}>0,\\
\mathfrak s_{A_y} & = \min\left\{\abs k_{\mathfrak s}-\alpha: k\in\N_0^d,\alpha\in A,\alpha<\gamma\right\}>0
\end{align*}
is necessary here.
\end{remark}

Next, we need an estimate in regimes of small scales, meaning for those $n$ such that $\norm{x-y}_{\mathfrak s}\geq 2^{-n}$.

\begin{lemma}
Let $n\in\N$ and the conditions of Lemma \ref{lemma:firstestimate} hold. Define
\begin{align*}
Y_n^\lambda(z) & \coloneqq\int_{\R^d}R_n(y-z)\psi_x^\lambda(y)dy,\\
Z_{n,l}^\lambda(z) & \coloneqq D^lR_n(x-z)\int_{\R^d}(y-x)^l\psi_x^\lambda(y)dy.
\end{align*}
Then, we have for $\tau\in  \mathcal{T}_{\alpha}$
\begin{align*}
\abs{(\Pi_x^t{\color{blue}\tau})(Y_n^\lambda)} & \leq C \norm{\Pi}_{\alpha;T}2^{-\beta n}\lambda^{\alpha-\beta_L},\\
\abs{(\Pi_x^t{\color{blue}\tau})(Z_{n,l}^\lambda)} & \leq C \norm{\Pi}_{\alpha;T}2^{(-\alpha-\beta+\abs l_{\mathfrak s}+\beta_L)n}\lambda^{\abs l_{\mathfrak s}},
\end{align*}
uniformly over $t\in[0,T]$, $x,y\in \R^d$ and $\norm{x-y}_{\mathfrak s}\geq 2^{-n}$.
\end{lemma}

\begin{proof}
To show this estimate we employ $R=LG$ to obtain
\begin{align*}
|Y_n^\lambda(z)|\leq \lambda^{-\beta_L} \left| \int_{\R^d}G_n(y-z)(L\psi)_x^\lambda(y)dy \right|.
\end{align*}
Here, we shifted a part of the singular behaviour of the kernel to the test-function, so that the convolution with $G_n$ gives the factor $ 2^{-\beta n}$. This step is crucial, since without any decay in $n$ on the right-hand side, summing up the $Y_n^\lambda(z)$ doesn't yield a helpful bound.\\
The estimate
\begin{align*}
	 \left| \int_{\R^d}G_n(y-z)(L\psi)_x^\lambda(y)dy \right|\leq  C \norm{\Pi}_{\alpha;T}2^{-\beta n}\lambda^{\alpha}
\end{align*}
can be deduced now in the same way as the proof of Lemma 5.19 in \cite{Hairer_2014} and for $Z_{n,l}^\lambda(z)$ we proceed similarly.
\end{proof}

Combining all these bounds as it has been done in Lemma 5.19 in \cite{Hairer_2014}, 
we finally obtain the estimate that shows the well-definedness of $\mathcal{J}$ and the right hand side of (\ref{RealizingModelIdentitiesR}).
\begin{lemma}
\begin{align*}
\abs{\sum_{n\geq 0}\int_{\R^d} (\Pi_x^t{\color{blue}\tau})(R^{0,\alpha}_{n,xy})\psi_x^\lambda(y)dy}
\leq C \norm{\Pi}_{\alpha;T}(1+\norm{\Gamma}_{\alpha;T})\lambda^{\alpha+\beta-\beta_L}.
\end{align*}
\end{lemma}

The Extension Theorem \ref{ExtensionTheorem} carries directly over to non-regularizing kernels, whenever $\mathcal{J}$ is well defined, so we 
can go directly to the actual lift of the convolution with $R$ to the level of regularity structures. Define for $\gamma>0$, $T>0$ and ${\color{blue}F}\in\mathscr{D}^{\gamma,\eta}(T;V;M)$
\begin{align}\label{DefinitionNR}
(\mathcal{N}{\color{blue}F})(t,x)\coloneqq\sum_{\abs k<\gamma}\sum_{n\geq 0}\frac{{\color{blue}X^k}}{k!}\langle\mathcal{R}_t{\color{blue}F}_t-\Pi_x^t{\color{blue}F}_t(x),D^k R_n(x-\cdot)\rangle,
\end{align}
where $t\in (0,T], x\in\R^d$.

\begin{theorem}\label{TheoremAbstractIntegrationR}
Given a normal regularity structure $\mathscr{T}=(A,\mathcal{T},G)$ equipped with an inhomogeneous model $M=(\Pi,\Gamma,\Sigma)$ and a sector $V\subset \mathcal{T}$ of regularity $\alpha_0$. Let $\alpha:=0\land \alpha_0$, $\gamma>0$ and $R$ be a non-regularizing kernel with $\gamma+\beta\notin\N $ and $N\geq \gamma+\beta-\beta_L\lor |\alpha|$. Further, let $\mathcal{I}:V\rightarrow \mathcal{T}$ be an abstract integration map of order $0$ and let $M$ realize $R$ for $\mathcal{I}$. For $\eta\in (\gamma-s_0,\gamma)$ and $T>0$ define the operator $\mathscr{R}$ for  ${\color{blue}F}\in\mathscr{D}^{\gamma,\eta}(T;V;M)$ by
\begin{align}\label{eq:DefAbstractR}
(\mathscr{R}{\color{blue}F})(t,x)=\mathcal{I}{\color{blue}F}(t,x)+\mathcal{J}(t,x){\color{blue}F}(t,x)+(\mathcal{N}{\color{blue}F})(t,x).
\end{align}
Then, $\mathscr{R}$ maps $\mathscr{D}^{\gamma,\eta}(T;V;M)$ onto $\mathscr{D}^{\gamma,\eta}(T;M)$ with
\begin{align*}
\triple \mathscr{R}{\color{blue}F}\triple _{\gamma ,\eta;T} \leq C \triple {\color{blue}F}\triple _{\gamma,\eta;T} \|\Pi\|_{\gamma;T} \|\Sigma\|_{\gamma;T}(1+\|\Gamma\|_{\gamma;T}+\|\Sigma\|_{\gamma;T})
\end{align*}
and the identity
\begin{align*}
\mathcal{R}_t\mathscr{R}{\color{blue}F}_t=(R\ast\mathcal{R}_{\cdot}{\color{blue}F}_{\cdot})(t)
\end{align*}
holds for $t\in (0,T]$.\\
If $\overline M=(\overline\Pi,\overline\Gamma,\overline\Sigma)$ is another inhomogeneous model realizing $R$ for $\mathcal{I}$ under the same assumptions as above and $\overline{\mathscr{R}}$ be defined as in \eqref{eq:DefAbstractR}, then we have for all ${\color{blue}F}\in\mathscr{D}^{\gamma,\eta}(T;V;M)$ and $\overline{{\color{blue}F}}\in\mathscr{D}^{\gamma,\eta}(T;V;\overline{M})$
\begin{align*}
\triple \mathscr{R}{\color{blue}F};\overline{\mathcal{K}}\overline{{\color{blue}F}}\triple _{\gamma ,\eta;T} \leq C (\triple {\color{blue}F};\overline{{\color{blue}F}}\triple _{\gamma,\eta;T} + \|M;\overline{M}\|_{\gamma;T}),
\end{align*}
where $C$ depends on $\triple {\color{blue}F}\triple _{\gamma,\eta;T},\triple \overline{{\color{blue}F}}\triple _{\gamma,\eta;T},\|M\|_{\gamma;T}$ and $\|\overline{M}\|_{\gamma;T}$.
\end{theorem}
\begin{proof}
The proof follows the lines of \cite{Hairer_2014}, see also \cite{hairer2017discretisations}. Here we only show the estimate
\begin{align*}
\norm{(\mathscr{R}{\color{blue}F})(t,x)}_{l}\leq C |t|_{0}^{(\eta-l) \land 0},
\end{align*}
where $l$ is an integer. The other parts of the proof can be 
carried over by similar adaptations.
Using the uniform in time estimates for our inhomogeneous model yields for $k=(k_1,\dots,k_n) \in \N_0^d$ with $|k|=l$
\begin{align*}
|\left<(\mathscr{R}_t{\color{blue}F}_t-\Pi_x^t{\color{blue}F}_t(x)),D^k R_n(x-\cdot)\right>|\leq C 2^{(l-\gamma)n}|t|_{0}^{\eta-\gamma}.
\end{align*}
Considering now $|t|_0\geq 2^{-(n+1)}$ we obtain by summing over those $n$
\begin{align*}
\sum_{n\in\N ,|t|_0\geq 2^{-(n+1)}} 2^{(|l|-\gamma)n}|t|_{0}^{\eta-\gamma}\leq C|t|_{0}^{\eta-l}
\end{align*}
and for $|t|_0< 2^{-(n+1)}$ we get
\begin{align*}
\sum_{n\in\N ,|t|_0< 2^{-(n+1)}} 2^{(l-\gamma)n}|t|_{0}^{\eta-\gamma}\leq C \sum_{n\in\N ,|t|_0< 2^{-(n+1)}} 2^{(l-\eta)n} \leq C |t|_{0}^{\eta-l}.
\end{align*}
Writing $\mathcal{J}(t,x){\color{blue}\tau}=\sum_{n\geq 0} \sum_{\abs k<\alpha}\frac{{\color{blue}X^k}}{k!} \langle\Pi_x^t{\color{blue}\tau},D^kR_n(x-\cdot)\rangle=:\sum_{n\geq 0}J^n(t,x){\color{blue}F}_t(x)$ we see that
\begin{align*}
\norm{J^n(t,x){\color{blue}F}_t(x)}_l \leq C\sum_{\zeta\in A, l<\zeta<\gamma} 2^{(l-\zeta)n}|t|_{0}^{(\eta-\zeta)\land 0}
\end{align*}
and from this we conclude as above
\begin{align*}
\norm{J(t,x){\color{blue}F}_t(x)}_l \leq C|t|_{0}^{\eta-l},
\end{align*}
which gives the desired estimate.
\end{proof}

\subsection{Functions with prescribed Singularities}
We will need some information about the behaviour of products and convolutions of smooth functions having a singularity of prescribed strength at the origin. Let us begin with a definition that captures this behaviour.

\begin{definition}\label{Def:Orderofkernel}
	Let $K:\R^{1+d}\setminus \{0\}\rightarrow\R$ be a smooth function and $\zeta\in\R$. Then $K$ is a \textit{kernel of order} $\zeta$ if there exists a constant $C$ such that for every $k\in\N_0^{1+d}$ the bound
	\begin{align*}
		\abs{D^k K(z)}\leq
		\begin{cases}
			C \norm{z}_{\mathfrak{s}}^{\zeta-|k|_{\mathfrak{s}}},\quad & \zeta-|k|_{\mathfrak{s}}\neq 0,\\
			C \abs{\log{\norm{z}_{\mathfrak{s}}}},\quad & \zeta-|k|_{\mathfrak{s}} =0
		\end{cases}
	\end{align*}
	holds for every $z=(t,x)\in\R^{d+1}$ with $\norm{z}_{\mathfrak{s}}\leq 1$.
\end{definition}

\begin{remark}
	\begin{itemize}
		\item[(i)] By definition, if $K$ is of order $\zeta$, then it is also of order $\bar\zeta$ for every $\bar\zeta <\zeta$.
		\item[(ii)] If $K$ is of order  $\zeta$, then $D^k$  is by definition of order $\zeta-|k|_{\mathfrak{s}}$.
	\end{itemize}
\end{remark}

For a proof of the next Lemma we refer to Lemma 10.14 in \cite{Hairer_2014}, see also Lemma 7.3 in \cite{hairer2017discretisations}.

\begin{lemma}\label{LemConvolutionKernels}
Let $K_1$ and $K_2$ be compactly supported kernels of orders $\zeta_1 \leq 0$ and $\zeta_2 \leq 0$, respectively. Then $K_1K_2$ is of order $\zeta_1+\zeta_2$. If $\zeta_1 \land \zeta_2 >-\abs{\mathfrak{s}}$ and $\bar\zeta:=\zeta_1+\zeta_2+\abs{\mathfrak{s}}\leq 0$, then $K_1\ast K_2$ is of order $\bar\zeta$.
\end{lemma}

As we are working with regularized versions $\xi_\varepsilon=\xi\ast\rho_\varepsilon$ of white noise, we will also encounter regularized versions of integral kernels of the form
\begin{align*}
	K_\varepsilon\coloneqq K\ast\rho_\varepsilon.
\end{align*}
The following lemma describes their behaviour and shows that they converge to the original kernel as the regularization gets removed. Again, for a proof we refer to \cite[Lemma 10.17]{Hairer_2014}.

\begin{lemma}\label{LemRegularKernels}
	Let $K$ be a kernel of order $\zeta\in(-\abs{\mathfrak{s}},0)$. Then, there exists a constant $C=C(K)$ such that one has the bound
	\begin{align}\label{EstimateRegularKernel}
		\abs{K_\varepsilon(z)}\leq C (\norm{z}_{\mathfrak{s}}+\varepsilon)^\zeta,
	\end{align}
	for all $z\in\R^{1+d}$. In particular, $K_\varepsilon$ is a kernel of order $\zeta$.\\
	Moreover, for all $\nu\in(0,1]$, one has the bound
	\begin{align}\label{EstimateRegularKernelDifference}
		\abs{K(z)-K_\varepsilon(z)}\leq C \varepsilon^\nu \norm{z}^{\zeta-\nu}_\mathfrak{s},
	\end{align}
	for all $z\in\R^{1+d}$.
\end{lemma}

For handling SQG we need later the order of the kernel which is given by the spatial convolution of the fractional heat kernel with the one of the Riesz transform. Using that this convolution for fixed $t>0$ decays as $|x|^{-d}$ and that it obeys the same scaling as the heat kernel, we obtain that it is also of the same order as the fractional heat kernel, see e.g. \cite{KOCH200122}.

\section{SQG}\label{s:sqg}
Now let us return to the SQG equation \eqref{IntroductionSQG} driven by a space-time white noise,
 i.e., in this section we consider,
\begin{equation}\label{SQG}
	\left\{
		\begin{array}{rll}
			\partial_t\theta+(-\Delta)^{\mu}\theta & = -R^\perp\theta\cdot\nabla\theta +\xi & \text{in } [0,\infty)\times\T^2,\\
			\theta(0,\cdot) & = \theta_0 & \text{in }\T^2,
		\end{array}
	\right.
\end{equation}
where $\theta_0:\T^2\rightarrow\R$ is the initial data.

We work with the scaling $\mathfrak{s}=(2\mu,1,1)$. Using that $R^\perp\theta$ is divergence free we obtain
\begin{align*}
		R^\perp\theta\cdot\nabla\theta= \div(\theta R^\perp \theta),
\end{align*}
hence,
\begin{align*}
	K\ast (R^\perp\theta\cdot\nabla\theta)= (\nabla K)\ast(\theta R^\perp \theta),
\end{align*}
where $K$ is the fractional heat kernel associated to the linear operator $\partial_t+(-\Delta)^{\mu}$. It is well-known that the fractional heat kernel in $d$-dimensions has the scaling property $K(t,x)=t^{-d/(2\mu)} K(1,t^{-1/(2\mu)}x)$ (see e.g. \cite{BerglundKuehn2017}).
Hence, it is a kernel of order $d=2$ in the sense of Definition \ref{Def:Orderofkernel} with respect to our scaling.  
The mild formulation of (\ref{SQG}) now reads
\begin{align}\label{SQGfixedPoint}
		\theta  =  (\nabla K)\ast\left(\theta R^\perp\theta\right)+K\ast \xi+K\ast_{\T^2}\theta_0.
\end{align}
The smoothing by the kernels  $\partial_j K$ are described by the parameter $\tilde{\mu}:=2\mu-1$, they are indeed $\tilde \mu$-smoothing since $K$ is $2\mu$-smoothing as it is shown in \cite{BerglundKuehn2017}.

Let us also make clear what we mean by space-time white noise on $\T^d$.

\begin{definition}
	Let $(\Omega,\mathscr{A},\mathbb{P})$ be a probability space. A random variable $\xi:\Omega\rightarrow\mathcal{D}'([0,\infty)\times\T^d)$ is called \textit{(Gaussian) space-time white noise} if
	\begin{itemize}
		\item[(i)] for each $\varphi\in\mathcal{D}([0,\infty)\times\T^d)$, the random variable $\omega\mapsto\langle\xi(\omega),\varphi\rangle$ is a Gaussian random variable,
		\item[(ii)] for all $\varphi,\psi\in\mathcal{D}([0,\infty)\times\T^d)$ one has
		\begin{align}\label{XiIsometry}
			\E[\langle\xi,\varphi\rangle\langle\xi,\psi\rangle]=(\varphi,\psi)_{L^2([0,\infty)\times\T^d)}.
		\end{align}
	\end{itemize}
	Therefore, the mapping $\varphi\mapsto\langle\xi,\varphi\rangle$ uniquely extends to a bounded linear operator from $L^2([0,\infty)\times\T^d)$ to $L^2(\Omega)$. Moreover, for any $\kappa>0$, it holds
	\begin{align*}
		\mathbb{P}\left(\xi\in C^{-\frac{(d+s_0)}{2}-\kappa}([0,\infty)\times\T^d)\right)=1,
	\end{align*}
	which is shown in \cite[Lemma 10.2]{Hairer_2014}, see also \cite{veraar2010}.
\end{definition}

To construct a regularity structure for SQG we proceed in the standard way, i.e., we first define a basis $\mathcal{F}$ of abstract symbols. The non-linearity is given by
$(\nabla K)\ast(\theta R^\perp\theta)= (\partial_1K) \ast (\theta R_2\theta)-(\partial_2K)\ast(\theta R_1\theta)$, which leads us to the following algorithm to 
construct the regularity structure.

We define ${\color{blue}\mathcal{I}[\xi]}$ as the symbol for $K\ast \xi$ (i.e., the solution to $\partial_t\theta+(-\Delta)^{\mu}\theta=\xi$ with zero initial data) and we denote by $\mathcal{F}_P:=\{{\color{blue}X^k}|k\in\N_0^{1+2}\}$ the basis of the polynomial regularity structure with ${\color{blue}X^0}=:{\color{blue}1}$, see Definition \ref{Def:PolynomialRegularityStructure}, and  $\mathcal{F}_0:=\overline{\mathcal{F}}_0:=\{{\color{blue}\mathcal{I}[\tau]}\}$, $\widetilde{\mathcal{F}}_{0}=\emptyset$.

Now, if we have a collection $\mathcal{F}_{n-1}=\overline{\mathcal{F}}_{n-1}\cup \widetilde{\mathcal{F}}_{n-1}$ of basis symbols given, we first construct the set $\widetilde{\mathcal{F}}_{n}$ by defining new symbols ${\color{blue}\mathcal{R}_i[\tau]}$, ${\color{blue}\mathcal{R}_i[\tau]\bar{\tau}}$, ${\color{blue}\mathcal{R}_i[\tau]X^k}$ and ${\color{blue}\tau X^k}$ for $i=1,2$,
${\color{blue}\tau}\in \overline{\mathcal{F}_{n-1}}$,
${\color{blue}\bar{\tau}}\in \overline{\mathcal{F}}_{l}$ for some $l\leq n-1$ and ${\color{blue}X^k}\in {\mathcal F}_P$. Here, we identify ${\color{blue}\tau 1}={\color{blue}\tau}$ and we set ${\color{blue}{\bar{\tau}\tau}}:={\color{blue}\tau\bar{\tau}}$ for any basis symbols $\tau,\bar\tau$.
Then, we define new symbols ${\color{blue}\mathcal{I}_i[\tau]}$ for $\tau\in \widetilde{\mathcal{F}}_{n}\cup \overline{\mathcal{F}}_{n-1}$ and these symbols form the set $\overline{\mathcal{F}}_{n}$.\\
Finally, we set $\mathcal{F}_{n}:=\widetilde{\mathcal{F}}_{n}\cup \overline{\mathcal{F}}_{n}$,  $\mathcal{F}_a:=\bigcup_{n=0}^{\infty}\mathcal{F}_{n}$  and  $\mathcal{F}:=\mathcal{F}_P\cup\mathcal{F}_{a}$.
Next, we fix some $\kappa\in(0,\tilde{\mu})$ and define the corresponding homogeneities by
\begin{align*}
	\abs{{\color{blue}\mathcal{I}[\Xi]}} & =-1+\mu-\kappa, 
	&\quad 
	\abs{{\color{blue}X^k}} & = \abs{k}_{\mathfrak{s}}, \\
	\abs{{\color{blue}\mathcal{I}_i[\tau]}} & = \abs{{\color{blue}\tau}}+\tilde{\mu},
	& \quad 
	\abs{{\color{blue}\mathcal{R}_i[\tau]}} & = \abs{{\color{blue}\tau}}, \\ 
	\abs{{\color{blue}\mathcal{R}_i[\tau]}{\color{blue}\bar{\tau}}} & =\abs{{\color{blue}\tau}}+\abs{{\color{blue}\bar{\tau}}},
	& \quad \abs{{\color{blue}\bar{\tau}}{\color{blue}X^k}} & =\abs{{\color{blue}\tau}}+\abs{k}_{\mathfrak{s}}
\end{align*}
and set $A$ to be the set of all homogeneities. Then $A$ is bounded from below and does not contain any accumulation point.
 We define the model space
\begin{align*}
	\mathcal{T} = \bigoplus_{\alpha\in A} \mathcal{T}_\alpha \coloneqq \bigoplus_{\alpha\in A} \langle{\color{blue}\tau}:{\color{blue}\tau}\in\mathcal{F},\ \abs{{\color{blue}\tau}}=\alpha\rangle
\end{align*}
and denote by $\mathcal{T}_P$ the sector generated by $\mathcal{F}_P$. Note that $\mathcal{T}_{\alpha}$ is finite dimensional for any $\alpha\in A$ by construction.

We only need to define the lift of the Riesz transforms $R_i$ on a substructure. Hence, we define the sector $V$ as the subspace of $\mathcal{T}$ generated by $\mathcal{F}_P$, ${\color{blue}\mathcal{I}[\Xi]}$ and all symbols ${\color{blue}\tau}\in \mathcal{F}_a$ of the form ${\color{blue}\tau}={\color{blue}\mathcal{I}_1[\bar{\tau}]}$ or $ {\color{blue}\tau}={\color{blue}\mathcal{I}_2[\bar{\tau}]}$ for some  ${\color{blue}\bar{\tau}}\in \mathcal{F}_a$. This is indeed a sector, provided that our models are admissible.\\
To describe the image of $R_i$ we define in a similar way the sector $W$ as the subspace of $\mathcal{T}$ generated by $\mathcal{F}_P$ and all symbols ${\color{blue}\tau}\in \mathcal{F}_a$ of the form ${\color{blue}\tau}={\color{blue}\mathcal{R}_1[\bar{\tau}]}$ or $ {\color{blue}\tau}={\color{blue}\mathcal{R}_2[\bar{\tau}]}$ for some  ${\color{blue}\bar{\tau}}\in \mathcal{F}_a$.

It will turn out, that we only need to consider functions taking values in $\mathcal{T}_{<\gamma}$ where we need choose $\gamma>1-\mu+\kappa$. Since the homogeneity of our symbols increases when choosing $\kappa$ smaller we can choose $\kappa,\gamma$ such that $\mathcal{T}_{<\gamma}=\mathcal{T}_{\leq 1- \mu}$. 

Let us introduce a handy graphical notation which gives a good overview of the abstract symbols in the model space $\mathcal{T}$. This notation is also independent of the indices $i$ and $j$ as the corresponding symbols share the same properties. An appearance of ${\color{blue}\Xi}$ is represented as a simple dot ${\color{blue}\cdot}$. The symbol that represents convolution with $K$, i.e. ${\color{blue}\mathcal{I}[\cdot]}$, is viewed as a straight line ${\color{blue}\RS{I}}$. Hence, the symbol ${\color{blue}\mathcal{I}[\Xi]}$ has the graphical representation ${\color{blue}\RS{i}}$. The symbol ${\color{blue}\mathcal{R}_i[\cdot]}$ associated to a component of the orthogonal Riesz transform is represented as a dotted line ${\color{blue}\RS{5}}$ and the symbol ${\color{blue}\mathcal{I}_i[\cdot]}$ associated to a derivative of the fractional heat kernel is viewed as zigzag line ${\color{blue}\RS{2}}$. The multiplication of two such trees is done by joining them at the root. For example, we have
\begin{align*}
{\color{blue}\mathcal{R}_i[\mathcal{I}[\Xi]]}={\color{blue}\RS{4i}}_i, \quad {\color{blue}
	\mathcal{R}_i[\mathcal{I}[\Xi]]
	[\mathcal{I}[\Xi]]}={\color{blue}\RS{{4i}{r}}}_i.
 \quad \text{and} \quad {\color{blue}\mathcal{I}_j[
 	\mathcal{R}_i[\mathcal{I}[\Xi]]
 	[\mathcal{I}[\Xi]]]}={\color{blue}\RS{2{4i}{r}}}_{ji}.
\end{align*}

It will also turn out, that all symbols containing a polynomial have a regularity strictly larger than $0$. We get, by dropping those symbols,
\begin{align*}
\overline{\mathcal F}_0= & \{{\color{blue}\RS{i}}\} \\
& -1-\kappa+ \mu\\
\tilde{\mathcal F}_1= &    \left\{ {\color{blue}\RS{4i}_i}, {\color{blue}\RS{{4i}r}_i}\right\} \\
& -1-\kappa+ \mu,-2-2\kappa+2 \mu \\
\overline{\mathcal F}_1 =&  \left\{ {\color{blue}\RS{24i}_{ji}}, {\color{blue}\RS{2{4i}r}_{ji}}\right\} \\
& -2-\kappa+3\mu,-3-2\kappa+4 \mu \\
\tilde{\mathcal F}_2= &    \left\{ {\color{blue}\RS{424i}}, {\color{blue}\RS{42{4i}r}}, {\color{blue}\RS{{424i}{r}}}, {\color{blue}\RS{{42{4i}r}{r}}},{\color{blue}\RS{{424i}{35r}}}, {\color{blue}\RS{{42{4i}r}{35r}}},{\color{blue}\RS{{424i}{3{4i}r}}}, {\color{blue}\RS{{42{4i}r}{3{l}{6i}}}}\right\}\\
& -2-\kappa+3\mu,
 -3-2\kappa+4\mu,
 -3-2\kappa+4\mu,
-4-3\kappa+5\mu,\\
&-4-2\kappa+6 \mu,
-5-3\kappa+7 \mu,
-5-3\kappa+7 \mu,
-6-4\kappa+8 \mu.
\end{align*}
We also dropped the indices in $\tilde{\mathcal F}_2$ for simplicity. We obtain that $V,W$ have regularity $-1-\kappa+\mu$ and $\min A=-2-2\kappa+2\mu$ (this minimum is the homogeneity of ${\color{blue}\RS{{4i}r}_i}$). Hence, any basis symbol containing a polynomial has at least the homogeneity of ${\color{blue}\RS{{4i}r}_i X^k}$, i.e., $2\mu-2-2\kappa+|k|_s$. For $k\neq 0$ we already have $2\mu-2+|k|_s>0$ since $\mu> 2/3$. This implies that if we choose $\kappa$ small enough, no symbols containing any polynomial appear in $\mathcal{T}_{<0}$.

The only missing piece in the regularity structure $\mathscr{T}_{SQG}$ describing the SQG equation is the structure group $G$. It can be constructed for the full model space $\mathcal{T}$ as in \cite[Theorem 8.24]{Hairer_2014}, we will not go into the details here.  Then $\mathscr{T}_{SQG}:=(A,\mathcal{T},G)$ is the regularity structure generated by SQG.\\

Next, we will lift any \textit{continuous} approximation $\xi_\varepsilon$ to the driving noise $\xi$ to a model $M^\varepsilon=(\Pi^\varepsilon,\Gamma^\varepsilon,\Sigma^\varepsilon)$ in a canonical way.
Therefore, let $\rho:\R^{1+2}\rightarrow\R$ be a symmetric mollifier and set $\rho_\varepsilon(t,x)\coloneqq\varepsilon^{-4}\rho(\varepsilon^{-2}t,\varepsilon^{-1}x)$. Then $\xi_\varepsilon\coloneqq\xi\ast\rho_\varepsilon$ is a smooth approximation to $\xi$ and we can define the canonical model $(\Pi^\varepsilon,\Gamma^\varepsilon,\Sigma^\varepsilon)$ as follows. First, for $t\in\R,\ x,y\in\R^2$ and $k=(k_0,\overline k)\in\N^{1+2}$, we set for the symbols in $\mathcal{F}_P$
\begin{align*}
(\Pi_x^{\varepsilon,t}{\color{blue}X^k})(y)\coloneqq\begin{cases} (y-x)^{\overline k}, & k_0=0 \\ 0 ,  & k_0>0.\end{cases}
\end{align*}
Then, we recursively define for ${\color{blue}\tau,\bar\tau}\in \mathcal{F}$ such that $\tau\bar\tau\in \mathcal{F}$ 
\begin{align*}
(\Pi_{x}^{\varepsilon,t}{\color{blue}\tau\bar\tau})(y)\coloneqq (\Pi_{x}^{\varepsilon,t}{\color{blue}\tau})(y)(\Pi_{x}^{\varepsilon,t}{\color{blue}\bar\tau})(y)
\end{align*}
and for $\tau\in\mathcal{F}_a$
\begin{align*}
(\Pi_{x}^{\varepsilon,t}{\color{blue}\mathcal{I}_j[\tau]})(y) 
\coloneqq 
&(\partial_j K\ast(\Pi_x^{\varepsilon,\cdot}\Sigma_x^{\cdot t}{\color{blue}\tau}))(t,y)\\
&- \sum_{\abs{\overline k}<\abs{{\color{blue}\tau}}+\tilde\mu, k_0=0} \frac{(y-x)^{\overline k}}{\overline k!}(D^{\overline k}\partial_j K\ast(\Pi_x^{\varepsilon,\cdot} \Sigma_x^{\cdot t}{\color{blue}\tau}))(t,x).
\end{align*}
Similarly, we define for $\tau\in\overline{\mathcal{F}}_n$
\begin{align*}
(\Pi_{x}^{\varepsilon,t}{\color{blue}\mathcal{R}_i[\tau]})(y) \coloneqq R_i \ast_{\T^2}(\Pi_{x}^{\varepsilon,t}{\color{blue}\tau})(y) - \sum_{\abs{\overline k}<\abs{{\color{blue}\tau}},k_0=0}\frac{(y-x)^{\overline k}}{\overline k!}(D^{\overline k}R_i)\ast_{\T^2}(\Pi_{x}^{\varepsilon,t}{\color{blue}\tau})(x)
\end{align*}
and $0=\Pi_{x}^{\varepsilon,t}{\color{blue}\mathcal{I}_j[\tau]}=\Pi_{x}^{\varepsilon,t}{\color{blue}\mathcal{R}_i[\tau]}$ for ${\color{blue}\tau}\in\mathcal{F}_P$. These maps are then (bi-)linearly extended to $\mathcal{T}$.
The maps $\Gamma_{xy}^{t,\varepsilon}$ and $\Sigma_x^{st,\varepsilon}$ are completely determined by the algebraic relations in the definition of a model. The analytical bounds follow as in \cite{Hairer_2014}.

\begin{remark}
In order to explicitly construct the mappings  $\Gamma_{xy}^{t,\varepsilon}$ and $\Sigma_x^{s,t,\varepsilon}$, one rather builds the mappings $\hat \Gamma_{(t,x)(s,y)}^{\varepsilon}$ for the homogeneous model first as it has been done in \cite{Hairer_2014}. Then, one sets $\Gamma_{xy}^{t,\varepsilon}:= \hat \Gamma_{(t,x)(t,y)}^{\varepsilon}$ and $\Sigma_x^{s,t,\varepsilon}:=\hat \Gamma_{(t,x)(s,x)}^{\varepsilon}$.
\end{remark}

\subsection{Abstract Fixed Point Equation}
In order to lift the SQG equation to an abstract fixed point equation in some $\mathcal{D}^{\gamma,\eta}(T;M)$ we assume that we are given a model $M$ and some constants $\gamma>0$, $\eta<0$. We will later on give the precise assumptions on those quantities and on $M$.\\
We define the abstract integration maps $\mathcal{I}_j$ of order $\tilde{\mu}$ for $\mathscr{T}_{SQG}$ by
\begin{align*}
\mathcal{I}_j{\color{blue}\tau}\coloneqq 
\begin{cases}
	{\color{blue}\mathcal{I}_j[\tau]}, & {\color{blue}\tau}\in \mathcal{F}\setminus\mathcal{F}_P,\\
 	0 & {\color{blue}\tau}\in \mathcal{F}_P,
\end{cases}
\end{align*}
and linearly extended to $\mathcal{T}$. It is then immediate that $\mathcal{I}_j$ satisfies the first two requirements of an abstract integration map while the third requirement follows from the general construction of the structure group.\\
In the following we assume that we have an admissible model $M=(\Pi,\Sigma,\Gamma)$ for $\mathscr{T}_{SQG}$ given. The canonical models $(\Pi^\varepsilon,\Gamma^\varepsilon,\Sigma^\varepsilon)$ built in the previous section satisfy this condition directly and also when renormalising them this holds true.\\
In order to define the lift $\mathcal{K}_j$ of the heat kernel $\partial_j K=:K_j$, we decompose it as
\begin{align*}
K_j=\tilde K_j+L_j^1,
\end{align*}
where $\tilde K_j$ is a kernel satisfying Assumption \ref{AssumptionK} for $\beta=\tilde{\mu}$ (see \cite{BerglundKuehn2017}) and $L_j^1$ is a smooth function. For  ${\color{blue}f}\in\mathcal{D}^{\gamma',\eta'}(T;V;M)$, where $0<\gamma'\leq \gamma$ and $\eta'\leq \eta\in\R$, $\bar\gamma>0$, we  define for $t\leq T$
\begin{align}\label{LiftOfSmoothPart}
\begin{split}
(\mathcal{L}_j^1{\color{blue}f})(t,x) 
&\coloneqq \sum_{\abs{k}_{\mathfrak{s}}\leq\bar\gamma}\frac{1}{k!}{\color{blue}X^k}(D^kL_j^1\ast\mathcal{R}_\cdot{\color{blue}f})(t,x)\\
&\coloneqq
\sum_{\abs{k}_{\mathfrak{s}}\leq\bar\gamma}\frac{1}{k!}{\color{blue}X^k}\int_0^t\mathcal{R}_s{\color{blue}f}_s(D^kL_j^1(t-s,x-\cdot))\ds,
\end{split}
\end{align}
where $\mathcal{R}$ denotes the corresponding reconstruction operator on $\mathcal{D}^{\gamma',\eta'}(T;M)$, and it is immediate to see that $\mathcal{L}_j^1:\mathcal{D}^{\gamma',\eta'}(T;M)\rightarrow\mathcal{D}^{\bar\gamma,\eta'}(T;V;M)$, since the function $L_j^1\ast\mathcal{R}_\cdot{\color{blue}f}$ is smooth.

Now, if $\eta'\in(\gamma'-2\mu,\gamma')$, by Theorem \ref{TheoremAbstractIntegration} we have that
\begin{align*}
\mathcal{K}_j:\mathscr{D}^{\gamma',\eta'}(T;M) & \rightarrow\mathscr{D}^{\gamma'+\tilde{\mu},(\eta'\wedge -2-2\kappa+2\mu)+\tilde{\mu}}(T;V;M),\\
(\mathcal{K}_j{\color{blue}f})(t,x) & \coloneqq \mathcal{I}_j{\color{blue}f}(t,x)+\mathcal{J}_j(t,x){\color{blue}f}(t,x)+(\mathcal{N}_j{\color{blue}f})(t,x),
\end{align*}
where $\mathcal{J}_j$ and $\mathcal{N}_j$ are given by (\ref{DefinitionJ}) and (\ref{DefinitionN}), respectively, is the lift of $\tilde K_j$.

To build an operator that represents the Riesz transform $R^\perp$ on the level of regularity structures we proceed in the same way and decompose $R_i$ as $R_i=\tilde R_i+L^2_i$, where $\tilde R_i$ satisfies Assumption \ref{AssumptionR}. $L^2_i$ can be lifted as $L^1_i$,
\begin{align}\label{LiftOfSmoothPartR}
\begin{split}
	(\mathcal{L}_i^2{\color{blue}f})(t,x) 
	& \coloneqq \sum_{\abs{k}_{\mathfrak{s}}\leq\bar\gamma}\frac{1}{k!}{\color{blue}X^k}(D^kL_i^2\ast_{\T^2}\mathcal{R}_t{\color{blue}f})(t,x)\\
	& \coloneqq
	\sum_{\abs{k}_{\mathfrak{s}}\leq\bar\gamma}\frac{1}{k!}{\color{blue}X^k}\mathcal{R}_t{\color{blue}f_t}(D^kL_i^2(,x-\cdot))\ds,
\end{split}
\end{align}
which gives $\mathcal{L}_2^1:\mathcal{D}^{\gamma',\eta'}(T;V;M)\rightarrow\mathcal{D}^{\gamma',\eta'}(T;W;M)$, see also the proof of Theorem \ref{TheoremAbstractIntegrationR}.

We set
\begin{align*}
	\mathcal{I}_{{R}_i}{\color{blue}\tau}\coloneqq
	\begin{cases}
		{\color{blue}\mathcal{R}_i[\tau]}, & {\color{blue}\tau}\in 
		\overline{\mathcal{F}}_n,n\in\N \\
		0 & \text{else},
	\end{cases}
\end{align*}
and linearly extended to $\mathcal{T}$.

Next, for models realizing $R_i$ we define $\mathcal{J}_{R_i}$ and $\mathcal{N}_{R_i}$ as in (\ref{DefinitionJR}) and (\ref{DefinitionNR}).

Now, if $\eta\in(\gamma-2\mu,\gamma)$, by Theorem \ref{TheoremAbstractIntegrationR} the lift of the Riesz transform is given by
\begin{align*}
\mathscr{R}_i:\mathscr{D}^{\gamma,\eta}(T;V;M) & \rightarrow\mathscr{D}^{\gamma,\eta}(T;W;M),\\
(\mathscr{R}_i{\color{blue}f})(t,x) & \coloneqq \mathcal{I}_{R_i}{\color{blue}f}(t,x)+\mathcal{J}_{R_i}(t,x){\color{blue}f}(t,x)+(\mathcal{N}_{R_i}{\color{blue}f})(t,x).
\end{align*}

Finally, we define for ${\color{blue}\tau_1}, {\color{blue}\tau_2}\in \mathcal{F}$ with ${\color{blue}\tau_1}={\color{blue}\mathcal{R}_i\bar{\tau}}$ for some ${\color{blue}\bar{\tau}} \in \mathcal{F}$ the abstract product by
\begin{align*}
	{\color{blue}\tau_1} \star {\color{blue}\tau_2}:=
	{\color{blue}\mathcal{R}_i\bar{\tau}}{\color{blue}\tau_2},
\end{align*}
${\color{blue}X^{k_1}}\star{\color{blue}X^{k_2}}\equiv {\color{blue}X^{k_1+k_2}}$ 
and bilinearly extended to all of $W \times \mathcal{T}$. 
By Theorem \ref{TheoremAbstractProduct} we have
for ${\color{blue}f_1}\in\mathscr{D}^{\gamma,\eta}_{-1-\kappa+\mu}(T;W;M)$, ${\color{blue}f_2}\in\mathscr{D}^{\gamma,\eta}_{-1-\kappa+\mu}(T;V;M)$ that
${\color{blue}f_1}\ast {\color{blue}f_1} \in \mathscr{D}^{\gamma',\eta'}_\alpha(T;M)$ with
\begin{align*}
	\alpha=-2-2\kappa+2\mu,\; 
	\gamma'=\gamma-1-\kappa+\mu, \; 
	\eta'=(\eta-1-\kappa+\mu)\land 2\eta.
\end{align*}

What is left is to determine $\gamma$ and $\eta$.\\ For ${\color{blue}\Theta}\in\mathscr{D}^{\gamma,\eta}(T;V;M)=\mathscr{D}^{\gamma,\eta}_{-1-\kappa+\mu}(T;V;M)$ with $\gamma>0$ and $\gamma-2\mu<\eta<\gamma$ we obtain
\begin{align*}
{\color{blue}\Theta}\star\mathscr{R}_i{\color{blue}\Theta}
\in\mathscr{D}^{\gamma-1-\kappa+\mu,(\eta-1-\kappa+\mu)\land 2\eta}_{-2-2\kappa+2\mu}(T;M).
\end{align*}
We need  $\gamma>1+\kappa-\mu$, which also implies $\gamma>0$ and $\gamma-1-\kappa-\mu<(\eta-1-\kappa+\mu)\land 2\eta<\gamma-1-\kappa+\mu$ to apply $\mathcal{K}_j$ to this function. We then obtain
\begin{align*}
	\mathcal{K}_j({\color{blue}\Theta}\star\mathscr{R}_i{\color{blue}\Theta})
	\in\mathscr{D}^{\gamma-2-\kappa+3\mu,((\eta-1-\kappa+\mu)\land 2\eta)+2\mu-1}_{-3-2\kappa+4\mu}(T;V;M).
\end{align*}
Aiming for a self-mapping property we need here $\gamma-2-\kappa+3\mu>\gamma$, $-3-2\kappa+4\mu>-1-\kappa+\mu$ and $((\eta-1-\kappa+\mu)\land 2\eta)+2\mu-1> \eta$. It follows from $\mu>2/3$ that all these inequalities are satisfied for $\kappa$ sufficiently small if we have $\eta>-2\mu+1=-\tilde\mu$ and choose $\gamma=1+2\kappa-\mu$.

With all these abstract linear operators at hand we can formulate the abstract version of the fixed point equation (\ref{SQGfixedPoint}) as follows. Given the regularity structure $\mathscr{T}_{SQG}$ together with any admissible model $M$. Let $\gamma>0$, $\eta<0$ and $\theta_0\in C^\eta(\T^2)$. Then by Lemma 3.6 in \cite{hairer2017discretisations} the function $S(t)\theta_0:=K(t,\cdot)\ast_{\T^2}\theta_0$ can be lifted to $\mathscr{D}^{\gamma,\eta}(T;V)$ such that $\triple S(\cdot)\theta_0 \triple_{\gamma,\eta;T} \leq c \|\theta_0\|_{C^\eta}$. We denote this lift by ${\color{blue}S}_\cdot\theta_0$ and consider the fixed point problem
\begin{align}\label{SQGfixedPointAbstract}
\begin{split}
	{\color{blue}\Theta}
	= & (\mathcal{K}_1+\mathcal{L}_1^1)P_{\leq 2-3\mu}({\color{blue}\Theta}\star (\mathscr{R}_2+\mathcal{L}_2^2){\color{blue}\Theta})
	-(\mathcal{K}_2+\mathcal{L}_2^1)P_{\leq 2-3\mu}({\color{blue}\Theta}\star(\mathscr{R}_1+\mathcal{L}_1^2){\color{blue}\Theta})\\
	&+{\color{blue}\mathcal{I}[\Xi]}+{\color{blue}S}_\cdot\theta_0
\end{split}
\end{align}
for ${\color{blue}\Theta}\in\mathscr{D}^{\gamma,\eta}(T;V;M)$, where $P_{\leq 2-3\mu}$ is the projection onto $\mathcal{T}_{\leq 2-3\mu}$.
Let us conclude this section with a theorem concerning existence and uniqueness of solutions.

\begin{theorem}\label{theprem:abstractsolvable}
Let $\mathscr{T}_{SQG}$ be the regularity structure for the SQG equation with $2/3< \mu \leq 1$ and let  $-2\mu+1<\eta<0$. 

Then, for $\gamma=1+2\kappa-\mu$ and $\kappa$ small enough there exists for every admissible model $M=(\Pi,\Gamma,\Sigma)$, and $\theta_0\in C^\eta(\T^2)$, a time $T'\in(0,\infty]$ such that, for every $T<T'$ the equation (\ref{SQGfixedPointAbstract}) admits a unique solution ${\color{blue}\Theta}\in\mathscr{D}^{\gamma,\eta}(T;V)$ in $[0,T)\times\T^2$. Furthermore, if $T'<\infty$, then
\begin{align*}
\lim_{T\rightarrow T'}\norm{\mathcal{R}_T\mathcal{S}_T(\theta_0,M)_T}_{C^\eta}=\infty,
\end{align*}
where $\mathcal{S}_T:(\theta_0,M)\mapsto{\color{blue}\Theta}$ is the solution map up to time $T$. Finally, for every $T<T'$, the solution map $\mathcal{S}_T$ is jointly Lipschitz continuous in a neighbourhood around $(\theta_0,M)$ in the following sense. Let $(\overline\theta_0,\overline M)$ be some initial data and define ${\color{blue}\overline\Theta}\coloneqq\mathcal{S}_T(\overline\theta_0,\overline Z)$, then for any $\varepsilon_0>0$ there is $L>0$ such that $\triple {\color{blue}\Theta};{\color{blue}\overline\Theta}\triple_{\gamma,\eta;T}\leq L\varepsilon$ whenever $\norm{\theta_0-\overline\theta_0}_{C^\eta}+\triple M;\overline M\triple_{\gamma;T}\leq\varepsilon$ for  $\varepsilon\in(0,\varepsilon_0]$.
\end{theorem}

\begin{proof}
This theorem can be seen as special case of Theorem 3.10 of \cite{hairer2017discretisations} combined with the estimates on $\mathscr{R}_i$ from Theorem \ref{TheoremAbstractIntegration} as well as the fact that, whenever ${\color{blue}\Theta}\in\mathcal{D}^{\gamma,\eta}_{-1-\kappa+\mu}(T;V;M)$, the right hand side of (\ref{SQGfixedPointAbstract}) belongs to $
\mathcal{D}^{\gamma,\eta}_{-1-\kappa+\mu}(T;V;M)$.
\end{proof}

A direct calculation yields that the solution ${\color{blue}\Theta}$ to (\ref{SQGfixedPointAbstract}) for any admissible model and $ \mu> 3/4$ and $\kappa$ small enough is given by
\begin{align}\label{SolutionFormula}
{\color{blue}\Theta}(t,x) = {\color{blue}\RS{i}} + {\color{blue}\RS{2{4i}{r}}}_{12}-{\color{blue}\RS{2{4i}{r}}}_{21}+ \Theta_{{\color{blue}1}}(t,x){\color{blue}1} 
 + r(t,x),
\end{align}
where
$r(z)$ is a modelled distribution with components of homogeneity greater than $\gamma$, meaning this part can be ignored for the description of $\mathcal{R}{\color{blue}\Theta}$.

In view of Theorem \ref{ReconstructionTheorem}(ii) the reconstruction operator maps $\mathscr{D}^{\gamma,\eta}(T;V)$ onto $C^{\bar{\delta},\alpha}_{\eta-\gamma;T}$ where $\alpha=-2+2\mu-2\kappa$ and $\bar{\delta}<-2+3\mu-\kappa$.

\subsection{Renormalization and Convergence of Models}
The canonical models constructed above do not converge in a suitable sense for $\varepsilon\to0$, hence we have to renormalize them, i.e., to find models $(\hat\Pi^\varepsilon,\hat\Gamma^\varepsilon,\hat\Sigma^\varepsilon)$ which converge.\\
In the following, we concentrate on the case for $\mu>4/5$ to keep the presentation short.\\
Therefore, let $\mathcal{F}_-$ be the set of basis symbols of negative homogeneity of the regularity structure $\mathscr{T}_{SQG}$. The key stochastic estimates are the following: 
We need to define the renormalised models in such a way that there are constants  $\nu,\nu',\delta>0$ such that for every ${\color{blue}\tau}\in\mathcal{F}_-$ there is random variable $\Pi^{t}_x{\color{blue}\tau}$ with
\begin{align}\label{KeyEstimate}
\E\left[\abs{\left(\Pi^{t}_x{\color{blue}\tau},\varphi_x^\lambda\right)}^2\right]
&\leq C \lambda^{2(\abs{{\color{blue}\tau}}+\nu)},\\
\label{KeyEstimateTime}
\E\left[\abs{\left(\Pi^{t}_x{\color{blue}\tau}-\Pi_x^{s}{\color{blue}\tau},\varphi_x^\lambda\right)}^2\right]
&\leq C \lambda^{2(\abs{{\color{blue}\tau}}+\nu-\delta)}|t-s|^{2\delta/s_0}
\end{align}
and
\begin{align}
\label{KeyEstimateDifference}
\E\left[\abs{\left(\hat\Pi^{\varepsilon,t}_x{\color{blue}\tau}-\Pi_x^t{\color{blue}\tau},\varphi_x^\lambda\right)}^2\right]
&\leq C' \varepsilon^{\nu'} \lambda^{2(\abs{{\color{blue}\tau}}+\nu)},\\
\label{KeyEstimateDifferenceTime}
\E\left[\abs{\left(\hat\Pi^{\varepsilon,t}_x{\color{blue}\tau}-\Pi_x^t{\color{blue}\tau}-\hat\Pi^{\varepsilon,s}_x{\color{blue}\tau}+\Pi_x^s{\color{blue}\tau},\varphi_x^\lambda\right)}^2\right]
&\leq C' \varepsilon^{\nu'} \lambda^{2(\abs{{\color{blue}\tau}}+\nu-\delta)}|t-s|^{2\delta/s_0}
\end{align}
uniformly over $\varepsilon,\lambda\in(0,1]$ and $\varphi\in\mathcal{B}_r$ for some $r>0$ and (locally) uniformly in $x$ and $t,s$.
Then the renormalised models converge to a model $(\Pi,\Gamma,\Sigma)$ with time regularity $\delta$ and we have $\mathcal{R}_{\varepsilon}{\color{blue}\Theta} \to \mathcal{R}{\color{blue}\Theta}$, where $\mathcal{R}_{\varepsilon}$ and $\mathcal{R}$ are the reconstruction operator of the models $(\hat\Pi^\varepsilon,\hat\Gamma^\varepsilon,\hat\Sigma^\varepsilon)$ resp. $(\Pi,\Gamma,\Sigma)$ and ${\color{blue}\Theta}$ is the solution of \eqref{SQGfixedPointAbstract}. Hence, $\mathcal{R}{\color{blue}\Theta}$ is the solution to the SQG equation driven by space-time white noise.  For the precise statement and its proof we refer to  \cite[Theorem 10.7]{Hairer_2014} and \cite[Remark 6.3]{hairer2017discretisations}.

Due to the local subcriticality of the SQG equation the number of symbols in $\mathcal{F}_-$ is finite. One can easily verify that 
\begin{align*}
\mathcal{F}_- = \left\{ {\color{blue}\RS{i}}, {\color{blue}\RS{4i}}, {\color{blue}\RS{{4i}{r}}}_1,{\color{blue}\RS{{4i}{r}}}_2 \right\},
\end{align*}
where we listed the symbols in increasing homogeneity. The homogeneities are
\begin{align*}
\ -1-\kappa+\mu,\ -1-\kappa+\mu,\ -2-2\kappa+2\mu,\ -2-2\kappa+2\mu.
\end{align*}
\begin{remark}
The fact that $\xi$ is Gaussian implies that we have to check only the second moments, hence we do not have to show the moment bounds (\ref{KeyEstimate}) - (\ref{KeyEstimateDifferenceTime}) for some $p>2$ which is a major simplification in our analysis. This is due to Nelson's Estimate, see \cite{NELSON1973211}.
\end{remark}

To prove the needed estimates we will use the Wiener chaos expansions. Here, we only briefly exemplify the idea and refer to Chapter 1 in \cite{Nua_2006} as well as Section 10.1 in \cite{Hairer_2014} for more details.\\
To illustrate the idea, let us consider the symbol ${\color{blue}\RS{{4i}{r}}}_i$ which belongs to $\mathcal{F}_-$. Because our aim is to show the estimates (\ref{KeyEstimate}) and (\ref{KeyEstimateDifference}), we begin with the following calculation. For any spatial test function $\varphi$ we have
\begin{align*}
\left(\Pi_0^{\varepsilon,t}{\color{blue}\RS{{4i}{r}}}_i,\varphi_0^\lambda\right)
& = \int_{\T^2}(R_i K\ast\xi_\varepsilon)(t,x)( K\ast\xi_\varepsilon)(t,x)\varphi_0^\lambda(x)dx\\
& = \int_{M_t}\int_{M_t}\left(\int_{\T^2}\varphi_0^\lambda(x)R_i K(t-t_1,x-x_1) K(t-t_2,x-x_2)dx\right)\\
& \qquad\qquad\qquad\xi_\varepsilon(t_1,x_1)\xi_\varepsilon(t_2,x_2)d(t_1,x_1)d(t_2,x_2)\\
& \eqqcolon \hat I_2\left(\int_{\T^2}\varphi_0^\lambda(x)R_i^\perp K(t-\cdot,x-\cdot) K(t-\cdot,x-\cdot)dx\right),
\end{align*}
where $M_T=(0,t)\times\T^2$.
This shows that $\left(\Pi_0^{\varepsilon,t}{\color{blue}\RS{{4i}{r}}}_i,\varphi_0^\lambda\right)$ can be understood as an iterated stochastic integral against $\xi_\varepsilon$. In view of (\ref{KeyEstimate}), one would ideally have an It\^o-isometry for such integrals or rather an estimate of the form
\begin{align*}
\E\left[\abs{\hat I_2(f)}^2\right]\leq \norm f^2_{L^2}.
\end{align*}
Unfortunately, this is not the case. The reason why is that the pointwise product $\xi_\varepsilon(z_1)\xi_\varepsilon(z_2)$, where $z_i=(t_i,x_i)$ is simply not the correct object to consider. However, we can overcome this problem by replacing it with the \textit{Wick product} of these terms, which is defined as
\begin{align*}
\xi_\varepsilon(z_1)\diamond\xi_\varepsilon(z_2) & \coloneqq \xi_\varepsilon(z_1)\xi_\varepsilon(z_2)-\E[\xi_\varepsilon(z_1)]\xi_\varepsilon(z_2)-\xi_\varepsilon(z_1)\E[\xi_\varepsilon(z_2)]-\E[\xi_\varepsilon(z_1)\xi_\varepsilon(z_2)] \\
& = \xi_\varepsilon(z_1)\xi_\varepsilon(z_2)-\E[\xi_\varepsilon(z_1)\xi_\varepsilon(z_2)]
\end{align*}
due to $\E[\langle\xi,\varphi\rangle]=0$. Then, defining the second order Wiener integral with respect to $\xi_\varepsilon$
\begin{align*}
I_2(f)\coloneqq \int_{M_t}\int_{M_t}f(z_1,z_2)(\xi_\varepsilon(z_1)\diamond\xi_\varepsilon(z_2))dz_1dz_2,
\end{align*}
we indeed have the estimate
\begin{align}\label{EstimateI}
\E\left[\abs{I_2(f)}^2\right]\leq \norm f^2_{L^2}.
\end{align}
Moreover, it is possible to construct distributions $I_2(f)$, in a similar way, when the regularized noise $\xi_\varepsilon$ gets replaced by $\xi$, see \cite{Nua_2006}. Such distributions will play an important role in constructing the limiting model $(\Pi,\Gamma,\Sigma)$.

\begin{remark}
The Wick powers and also the operators $I_k$ which represent the $k$-th fold iterated stochastic integral are defined for any $k\in\N$, see \cite[Section 10.1]{Hairer_2014}. For $\mu>4/5$, we only need $I_2$ and the simpler operator $I_1$ that represents integration against white noise. 
\end{remark}

It remains to see how to transform $\left(\Pi_x^{\varepsilon,t}{\color{blue}\RS{{4i}{r}}}_i,\varphi_x^\lambda\right)$ so that we are in the correct setting of iterated integrals against the Wick product of $\xi_\varepsilon$ or $\xi$ in the sense of distributions.

\subsection{Renormalization}
With the theory of Wiener chaos expansions at hand, we can turn to the renormalization of the canonical model $\Pi^\varepsilon$ into a convergent model $\hat\Pi^\varepsilon$. 
In detail, the aim is to construct a new model $\hat\Pi^\varepsilon$ and to find functions $\left(\mathcal{W}^{(k;\varepsilon)}{\color{blue}\tau}\right)((t,y);\cdot)\in L^2(\R^{1+2})^{\otimes k}$ for any ${\color{blue}\tau}\in\mathcal{F}_-$ 
such that
\begin{align}\label{DecompPi}
\left(\hat\Pi_x^{\varepsilon,t}{\color{blue}\tau},\varphi_x^\lambda\right)  = I_2\left(\int_{\T^2}\varphi_x^\lambda(y)S_{x}^{\otimes k}\left(\mathcal{W}^{(k;\varepsilon)}{\color{blue}\tau}\right)((t,y);\cdot)dy\right),
\end{align}
where $k\in\{1,2\}$, depending on the number of appearances of ${\color{blue}\Xi}$ in ${\color{blue}\tau}$, and $S_{x}^{\otimes k}$ is the translation by $x$ acting on $L^2(\R^{1+2})^{\otimes k}$. We are in the translation invariant situation, meaning that 
$\xi_{\varepsilon}(x,t)=I_1(S_x^{\otimes 1}\rho_\varepsilon(\cdot,t-\cdot))$.\\
Recall that in case of $k=1$ no iterated stochastic integrals appear and $I_1(\cdot)$ corresponds to the integral against $\xi_\varepsilon$ or, in a distributional sense, against $\xi(\cdot)$ and the function, which gets evaluated by $\xi$, depends on $\varepsilon$. This means that for every symbol ${\color{blue}\tau}$ that contains only one appearance of ${\color{blue}\Xi}$ no renormalization is needed because $\left(\Pi_x^{\varepsilon,t}{\color{blue}\tau},\varphi_x^\lambda\right)$ has already in the desired form. The functions $\left(\mathcal{W}^{(1;\varepsilon)}{\color{blue}\tau}\right)((t,y);\cdot)$ for these symbols are given by
\begin{align*}
\left(\mathcal{W}^{(1;\varepsilon)}{\color{blue}\RS{i}}\right)((t,y);(t_1,y_1)) & \coloneqq  K_\varepsilon(t-t_1,y-y_1),\\
\left(\mathcal{W}^{(1;\varepsilon)}{\color{blue}\RS{4i}_i}\right)((t,y);(t_1,y_1)) & \coloneqq R_i^\perp K_\varepsilon(t-t_1,y-y_1).
\end{align*}
To see how to define $\hat\Pi^\varepsilon{\color{blue}\RS{{4i}{r}}}_i$, we proceed as follows. Setting $z=(t,x)$ and similar for $z_i,\bar z, z'_i$ the straightforward but lengthy calculation
\begin{align*}
&\int_{M_t}\int_{M_t}  R_i K((t-t_1,y-y_1)) K((t-t_2,y-y_2))\xi_\varepsilon(t_1,y_1)\diamond\xi_\varepsilon(t_2,y_2)d(\!t_1,y_1\!)d(\!t_2,y_2\!)\\
& \underset{\mathclap{(\ref{XiIsometry})}}{=} \quad \int_{M_t}\int_{M_t} R_i K(z-z_1) K(z-z_2)\xi_\varepsilon(z_1)\xi_\varepsilon(z_2)dz_1dz_2\\
& \quad - \int_{M_t}\int_{M_t} R_i K(z-z_1) K(z-z_2)\left(\int_{M_t}\rho_\varepsilon(z_1-\bar z)\rho_\varepsilon(z_2-\bar z)d\bar z\right)dz_1dz_2\\
& = (\Pi^{\varepsilon,t}_x{\color{blue}\RS{{4i}{r}}}_i)(z)
- \int_{M_t} \int_{M_t} R_i K(z-z_1)\rho_\varepsilon(z_1-\bar z)dz_1 \int_{M_t} K(z-z_2)\rho_\varepsilon(z_2-\bar z)dz_2 d\bar z\\
& = (\Pi^{\varepsilon,t}_x{\color{blue}\RS{{4i}{r}}}_i)(z)
- \int_{M_t} \int_{M_t} R_i K(z_1')\rho_\varepsilon(z-\bar z-z_1')dz_1' \int_{M_t} K(z_2')\rho_\varepsilon(z-\bar z-z_2')dz_2' d\bar z\\
& = (\Pi^{\varepsilon,t}_x{\color{blue}\RS{{4i}{r}}}_i)(z) - \int_{M_t}  R_i
 K_\varepsilon(\bar z)  K_\varepsilon(\bar z)d\bar z
\end{align*}
yields that
\begin{align}\label{DefRenormModel1}
\left(\hat\Pi^{\varepsilon,t}_x{\color{blue}\RS{{4i}{r}}}_i\right)(y) \coloneqq \left(\Pi^{\varepsilon,t}_x{\color{blue}\RS{{4i}{r}}}_i\right)(y) - C^i_\varepsilon,
\end{align}
where $C^i_\varepsilon\coloneqq (R_i K_\varepsilon, K_\varepsilon)$ is the object to consider, and we define
\begin{align}\label{SymbolRieszNabla}
\left(\mathcal{W}^{(2;\varepsilon)}{\color{blue}\RS{{4i}{r}}}_i\right)((t,y);(t_1,y_1);(t_2,y_2)) & \coloneqq R_i K_\varepsilon(t-t_1,y-y_1) K_\varepsilon(t-t_2,y-y_2)
\end{align}
and $\left(\mathcal{W}^{(1;\varepsilon)}{\color{blue}\RS{{4i}{r}}}_i\right)((t,y);(t_1,y_1);(t_2,y_2))=0$.
Let us mention that by symmetry we have $C_\varepsilon^1=C_\varepsilon^2$.

This shows that we can define the renormalized model $(\hat\Pi^\varepsilon,\hat\Gamma^\varepsilon,\hat\Sigma^\varepsilon)$ by the identity (\ref{DefRenormModel1}) and for the remaining symbols in $T$, $\hat\Pi^\varepsilon$ simply acts as $\Pi^\varepsilon$. Due to the algebraic constraint of a model, the maps $\hat\Gamma^\varepsilon$ and $\hat\Sigma^\varepsilon$ are uniquely characterized by $\hat\Pi^\varepsilon$.

\begin{remark}
Using the scaling properties of $K$ and $R_i K$ together with Lemma \ref{LemRegularKernels}, it is easy to see that the constants $C^i_\varepsilon$ diverge like $\varepsilon^{-2+2\mu}$ for $\mu<1$ resp. like $\log(\varepsilon)$ for $\mu=1$.
\end{remark}

\begin{remark}\label{RemarkRenormGroup}
The entire renormalization procedure we have done here is a special case of a general approach, which is presented in Section 8.3 of \cite{Hairer_2014}. The idea is to introduce a renormalization group $\mathfrak{R}$ that acts on the space of admissible models.
Such elements $M\in\mathfrak{R}$ are linear maps from the model space $\mathcal{T}$ to itself and an important feature is their action on elements of negative homogeneity. In our case, this resembles the fact that ${\color{blue}\RS{{4i}{r}}_i}$ has to be renormalized by subtracting a suitable constant. For each $M$, given any admissible model $\Pi$, one can build a new model $\Pi^M$ with the property
\begin{align*}
\Pi^M = \Pi M.
\end{align*}
The renormalization procedure that we have done here can then be viewed as such a transformation using the map $M^\varepsilon$ defined by
\begin{align*}
M^\varepsilon{\color{blue}\RS{{4i}{r}}}_i={\color{blue}\RS{{4i}{r}}}_i-C^i_\varepsilon{\color{blue}1}, 
\end{align*}
and $M^\varepsilon{\color{blue}\tau}={\color{blue}\tau}$ for the remaining symbols ${\color{blue}\tau}\in \mathcal{T}$\\
Moreover, the crucial property of the renormalization group is that it is an invariant mapping on the space of admissible models, meaning every model of the form $\Pi^M$, where $\Pi$ is admissible, is again admissible, see \cite{Hairer_2014}.
\end{remark}

For ${\color{blue}\hat\Theta}$ being the solution to the abstract fixed point problem (\ref{SQGfixedPointAbstract}) with model $\hat\Pi^\varepsilon$ the function $\hat{\mathcal{R}}_t^\varepsilon{\color{blue}\hat{\Theta}}_t$ solves the equation
\begin{align*}
\hat{\mathcal{R}}_t^\varepsilon{\color{blue}\hat\Theta}_t & = \partial_1 K\ast R_2\left(\hat\Pi_{\cdot}^\varepsilon{\color{blue}\hat\Theta}\right) - \partial_2 K\ast R_1\left(\hat\Pi_{\cdot}^\varepsilon{\color{blue}\hat\Theta}\right)
 +K\ast \xi_\varepsilon+K \ast_{\T^2} \theta_0\\
\end{align*}
Hence, the function $\hat{\mathcal{R}}_t^\varepsilon{\color{blue}\hat\Theta}_t$ is the mild solution of the SQG equation with regularized noise. Note, that the constant $C_\varepsilon^i$ do not appear in this equation.

\subsection{Convergence of Renormalized Models}
Given the preliminaries of the previous subsections, we are now in a convenient position to show the convergence of $\hat\Pi^\varepsilon$ to some $\Pi$  as $\varepsilon\rightarrow 0$. This then implies the existence of a solution to the renormalized and regularized version of the SQG equation as the regularization $\varepsilon$ gets removed.
Instead of showing the estimates (\ref{KeyEstimate}) - (\ref{KeyEstimateDifferenceTime}) directly, we work with the functions $ \left(\mathcal{W}^{(k;\varepsilon)}{\color{blue}\tau}\right)(\bar t,\bar x)$.
Since the following proposition is of general interest, we formulate it in $\R^d$ with the scaling $(s_0,1,\dots,1)$, our application will be the case $d=2$ and $s_0=2\mu$.

\begin{proposition}\label{PropositionEstimateW}
Let ${\color{blue}\tau}\in\mathcal{F}_-$ and $k\in\N$ be less or equal the number of appearances of ${\color{blue}\Xi}$ in ${\color{blue}\tau}$. Assume there exist $\nu,\nu'>0$ with $2(|\tau|+\nu)>-d$, functions $\left(\mathcal{W}^{(k)}{\color{blue}\tau}\right)(t,x)\in L^2(\R^{1+d})^{\otimes k}$ and a constant $C>0$ such that
\begin{align}\label{EstimateW}
\abs{\left\langle\left(\mathcal{W}^{(k)}{\color{blue}\tau}\right)(t,x),\left(\mathcal{W}^{(k)}{\color{blue}\tau}\right)( t,\bar x)\right\rangle}
 \leq C \sum_{\zeta}(|x|+|\bar x| )^{\zeta}
  |x-\bar x|^{2(\abs{{\color{blue}\tau}}+\nu)-\zeta}
\end{align}
and
\begin{multline}\label{EstimateWDiff}
\abs{\left\langle\left(\delta\mathcal{W}^{(k;\varepsilon)}{\color{blue}\tau}\right)(t,x),\left(\delta\mathcal{W}^{(k;\varepsilon)}{\color{blue}\tau}\right)(t,\bar x)\right\rangle} \\
\leq C \varepsilon^{\nu'}\sum_{\zeta} (|x|+|\bar x| )^{\zeta}
|x-\bar x|^{2(\abs{{\color{blue}\tau}}+\nu)-\zeta}
\end{multline}
hold, where the sums run over finitely many $\zeta\in[0,2(\abs{{\color{blue}\tau}}+\nu)+d)$, $\delta\mathcal{W}^{(k;\varepsilon)}\coloneqq\mathcal{W}^{(k;\varepsilon)}-\mathcal{W}^{(k)}$ and $\langle\cdot,\cdot\rangle$ denotes the scalar product in $L^2(\R^{1+d})^{\otimes k}$.\\
Then, the bounds (\ref{KeyEstimate}) and (\ref{KeyEstimateDifference}) are satisfied for ${\color{blue}\tau}$, where
\begin{align}\label{definitionPiTau}
	(\Pi_t^x{\color{blue}\tau })(\varphi):= \sum_{k} I_k\left(\int_{\T^2}\varphi(y)S_{x}^{\otimes k}\left(\mathcal{W}^{(k)}{\color{blue}\tau}\right)(t,y)dy\right).
\end{align}
If in addition
\begin{multline}\label{EstimateWTime}
	\abs{\left\langle\left(\mathcal{W}^{(k)}{\color{blue}\tau}\right)(t,x)-\left(\mathcal{W}^{(k)}{\color{blue}\tau}\right)(s,x),\left(\mathcal{W}^{(k)}{\color{blue}\tau}\right)( t,\bar x)-\left(\mathcal{W}^{(k)}{\color{blue}\tau}\right)(s,\bar x)\right\rangle} \\
	\leq C \sum_{\zeta}(\norm{t-s,x}_{\mathfrak{s}}+\norm{t-s,\bar x}_{\mathfrak{s}} )^{\zeta}
	|x-\bar x|^{2(\abs{{\color{blue}\tau}}+\nu-\delta)-\zeta} |t-s|^{2\delta/s_0}
\end{multline}
and
\begin{multline}\label{EstimateWDiffTime}
	\abs{\left\langle\left(\delta \mathcal{W}^{(k)}{\color{blue}\tau}\right)(t,x)-\left(\delta \mathcal{W}^{(k)}{\color{blue}\tau}\right)(s,x),\left(\delta \mathcal{W}^{(k)}{\color{blue}\tau}\right)( t,\bar x)-\left(\delta \mathcal{W}^{(k)}{\color{blue}\tau}\right)(s,\bar x)\right\rangle} \\
\leq C \varepsilon^{\nu'}\sum_{\zeta}(\norm{(t-s,x)}_{\mathfrak{s}}+\norm{(t-s,\bar x)}_{\mathfrak{s}} )^{\zeta}
|x-\bar x|^{2(\abs{{\color{blue}\tau}}+\nu-\delta)-\zeta} |t-s|^{2\delta/s_0}
\end{multline}
hold for some $\delta\in(0,\abs{{\color{blue}\tau}}+\nu+d/2)$ and finitely many $\zeta\in[0,2(\abs{{\color{blue}\tau}}+\nu+\delta)+d)$, then the bounds (\ref{KeyEstimateTime}) and (\ref{KeyEstimateDifferenceTime}) are also satisfied.
\end{proposition}

This result for homogeneous models is Proposition 10.11 in \cite{Hairer_2014}. For the  adaptations to the case of inhomogeneous models, see Proposition 6.2 and Remark 6.3 in \cite{hairer2017discretisations}. Since we have chosen and slightly different formulation than in those papers, let us give a short proof.

\begin{proof}
Let $\lambda \in(0,1]$. By \cite{Nua_2006} we have $\E[I_k(f)^2]\leq \|f\|_{ L^2(\R^{1+d})^{\otimes k}}^2$, hence
\begin{align*}
\E & \left|  I_k\left(\int_{\T^2}\varphi_x^\lambda (y)S_{x}^{\otimes k}\left(\mathcal{W}^{(k)}{\color{blue}\tau}\right)(t,y) dy \right)\right|^2\\
& \leq \left\| \int_{\R^d}  \varphi_0^\lambda (y) \left(\mathcal{W}^{(k)}{\color{blue}\tau}\right)( t, y)  dy \right\|^2_{ L^2(\R^{1+d})^{\otimes k}}\\
& = \int_{\R^d}\int_{\R^d}   \varphi_0^\lambda (y)  \varphi_0^\lambda (\bar y) \left\langle\left(\mathcal{W}^{(k)}{\color{blue}\tau}\right)(t,y),\left(\mathcal{W}^{(k)}{\color{blue}\tau}\right)( t,\bar y)\right\rangle  dy  d\bar y \\
& \leq c \lambda^{-2d} \sum_{\zeta}\int_{|y|\leq \lambda}\int_{|\bar y|\leq \lambda }(|y|+|\bar y| )^{\zeta} |y-\bar y|^{2(\abs{{\color{blue}\tau}}+\nu)-\zeta} dy  d\bar y \\
& \leq c \lambda^{2(\abs{{\color{blue}\tau}}+\nu)},
\end{align*}
This shows \eqref{KeyEstimate}, \eqref{KeyEstimateDifference} follows in the same way. For $|t-s|^{1/s_0}\geq \lambda$ we also have \eqref{KeyEstimateTime} and \eqref{KeyEstimateDifferenceTime} by the above calculation, but to show it for $|t-s|^{1/s_0}< \lambda$ we need the additional assumptions. Proceeding in the same way as before we get
\begin{align*}
	\E & \left|  I_k\left(\int_{\T^2}\varphi_x^\lambda (y)S_{x}^{\otimes k}\left(\left(\mathcal{W}^{(k)}{\color{blue}\tau}\right)(t,y)-\left(\mathcal{W}^{(k)}{\color{blue}\tau}\right)(s,y)\right) dy \right)\right|^2\\
	& \leq c \lambda^{-2d} 
	\sum_{\zeta}\int_{|y|\leq \lambda}
	\int_{|\bar y|\leq \lambda }
	(\norm{(t-s,y)}_{\mathfrak{s}} +\norm{(t-s,\bar y}_{\mathfrak{s}} )^{\zeta}\\
	&\qquad\qquad\qquad	 |y-\bar y|^{2(\abs{{\color{blue}\tau}}+\nu-\delta)-\zeta} 
		 |t-s|^{2\delta/s_0} dy  d\bar y \\
	& \leq c \lambda^{2(\abs{{\color{blue}\tau}}-\delta+\nu)}|t-s|^{2\delta/s_0}.
\end{align*}
\end{proof}

Note that the condition $2(|\tau|+\nu)>-d$ is satisfied for all basis symbols used to construct $\mathscr{T}_{SQG}$ if we choose $\kappa$ small enough, since $2|\tau|+2>-2+4\mu-4\kappa$

In many applications, including our case of SQG with noise homogeneous in time and space, we have $\left\langle\left(\mathcal{W}^{(k)}{\color{blue}\tau}\right)(\bar t,x),\left(\mathcal{W}^{(k)}{\color{blue}\tau}\right)( t,\bar x)\right\rangle=H(t-\bar t,x-\bar x)$. In this case, we only need $H$ to be a kernel of order $2(\abs{{\color{blue}\tau}}+\nu) $  to have \eqref{EstimateW} and \eqref{EstimateWTime}.

\begin{corollary}\label{corollaryEstimateW}
	In the situation of Proposition \ref{PropositionEstimateW}
	assume that  $\tilde H(t,\bar t, x,\bar x)
	:=\left\langle\left(\mathcal{W}^{(k)}{\color{blue}\tau}\right)(\bar t,x),\left(\mathcal{W}^{(k)}{\color{blue}\tau}\right)( t,\bar x)\right\rangle$ be a is a kernel of order $2(\abs{{\color{blue}\tau}}+\nu)\leq 0$ for some $\nu>0$ and only depend on $(t-\bar t,x-\bar x)$, i.e. $\tilde H(t,\bar t, x,\bar x)=H(t-\bar t, x-\bar x)$ for some kernel $H$ of the same order. Then the bounds \eqref{EstimateW} and \eqref{EstimateWTime} hold for $\delta \in [0,1/2]$.
\end{corollary}
\begin{proof}
The estimate \eqref{EstimateW} follows from
\begin{align*}
	\abs{\left\langle\left(\mathcal{W}^{(k)}{\color{blue}\tau}\right)(t,x),\left(\mathcal{W}^{(k)}{\color{blue}\tau}\right)( t,\bar x)\right\rangle}
	=H(0,x-\bar x)  
	\leq C 
	|x-\bar x|^{2(\abs{{\color{blue}\tau}}+\nu)}.
\end{align*}
For \eqref{EstimateWTime} we use
\begin{align*}
&\left\langle\left(\mathcal{W}^{(k)}{\color{blue}\tau}\right)(t,x)-\left(\mathcal{W}^{(k)}{\color{blue}\tau}\right)(s,x),\left(\mathcal{W}^{(k)}{\color{blue}\tau}\right)( t,\bar x)-\left(\mathcal{W}^{(k)}{\color{blue}\tau}\right)(s,\bar x)\right\rangle \\
&\qquad = H(0,x-\bar x)-H(s-t,x-\bar x)-H(t-s,x-\bar x) +H(0,x-\bar x).
\end{align*}
By \cite{Hairer_2014}, Lemma 10.18 we have for any $\alpha\in[0,1]$
\begin{multline*}
|H(y,r)-H(\bar y,\bar r)|\leq C \norm{(y-\bar y,r-\bar r)}_{\mathfrak{s}}^{\alpha} (\norm{(y,r)}_{\mathfrak{s}}^{2(\abs{{\color{blue}\tau}}+\nu)-\alpha}+\norm{(\bar y,\bar r)}_{\mathfrak{s}}^{2(\abs{{\color{blue}\tau}}+\nu)-\alpha} )
\end{multline*}
for any $(\bar y,\bar r),(\bar y,\bar r)$. This yields
\begin{align*}
	&\left|\left\langle\left(\mathcal{W}^{(k)}{\color{blue}\tau}\right)(t,x)-\left(\mathcal{W}^{(k)}{\color{blue}\tau}\right)(s,x),\left(\mathcal{W}^{(k)}{\color{blue}\tau}\right)( t,\bar x)-\left(\mathcal{W}^{(k)}{\color{blue}\tau}\right)(s,\bar x)\right\rangle\right| \\
	&\qquad \leq  C \norm{(t-s,0)}_{\mathfrak{s}}^{\alpha} \norm{(t-s,x-\bar x)}_{\mathfrak{s}}^{2(\abs{{\color{blue}\tau}}+\nu)-\alpha}\\
	&\qquad \leq C |t-s|^{\alpha/s_0} |x-\bar x|^{2(\abs{{\color{blue}\tau}}+\nu)-\alpha}
\end{align*}
for $\delta =\alpha/2 \in[0,1/2]$, where we used that $2(\abs{{\color{blue}\tau}}+\nu)-\alpha\leq 0$.
\end{proof}

Finally, we have everything set up to proof the convergence of the renormalised models.

\begin{theorem}
Given the regularity structure $\mathscr{T}_{SQG}$ for the SQG equation. Let $\rho$ be a symmetric mollifier, set $\xi_\varepsilon=\xi\ast\rho_\varepsilon$ and let $(\Pi^\varepsilon,\Gamma^\varepsilon,\Sigma^\varepsilon)$ be the associated canonical model. Further, let $(\hat\Pi^\varepsilon,\hat\Gamma^\varepsilon,\hat\Sigma^\varepsilon)$ be the renormalized model constructed above.\\
Then, there exists a unique admissible model $(\Pi,\Gamma,\Sigma)$ independent of the choice of mollifier $\rho$ such that $(\hat\Pi^\varepsilon,\hat\Gamma^\varepsilon,\hat\Sigma^\varepsilon)$ converges in probability to $(\Pi,\Gamma,\Sigma)$. Furthermore, $(\Pi,\Gamma,\Sigma)$ has time regularity $\delta>0$ for any $\delta \in (0,2 \mu -1)$.
\end{theorem}

\begin{proof}
It turns out that we can define functions $\mathcal{W}^{(k)}{\color{blue}\tau}$ in the same way we decomposed $\Pi^{\varepsilon,t}_{x}{\color{blue}\tau}$ in (\ref{DecompPi}) as we will see below. 
Once this is done, all that is left to show are the bounds (\ref{EstimateW})-(\ref{EstimateWDiff}) and (\ref{EstimateWTime})-(\ref{EstimateWDiffTime}) for $\mathcal{W}^{(k;\varepsilon)}$ and $\mathcal{W}^{(k)}$. 

From the symbols in $\mathcal{F}_-$ only ${\color{blue}\RS{{4i}{r}}}_i$ take some effort, since these are the symbols with $k=2$. Both symbols have homogeneity $|{\color{blue}\RS{{4i}{r}}}_i|=-2+2 \mu-2\kappa$. Due to their similar structure we drop the index $i$ for simplicity. Once we have done this, it is clear that the remaining symbols with $k=1$ can be dealt with in the same fashion with even less effort because the corresponding functions $\left(\mathcal{W}^{(k;\varepsilon)}{\color{blue}\tau}\right)(t,x)$ and $\left(\mathcal{W}^{(k)}{\color{blue}\tau}\right)(t,x)$ are easier to work with. 
We define
\begin{align*}
	\left(\mathcal{W}^{(2)}{\color{blue}\RS{{4i}{r}}}\right)(z;z_1;z_2) & \coloneqq R_i K(z-z_1) K(z-z_2)
\end{align*}
and $\left(\mathcal{W}^{(1)}{\color{blue}\RS{{4i}{r}}}\right)(z;z_1;z_2)  \coloneqq 0$,
where $z=(t,x),$ and $z_i,\overline{z}$ are defined similarly. Hence,
\begin{align*}
& \left\langle\left(\mathcal{W}^{(2)}{\color{blue}\RS{{4i}{r}}}\right)(z),\left(\mathcal{W}^{(2)}{\color{blue}\RS{{4i}{r}}}\right)(\bar z)\right\rangle\\
& \quad = \int_{M_t}\int_{M_t} R_i K(z-z_1) K(z-z_2) R_i K(\bar z-z_1) K(\bar z-z_2) dz_1dz_2\\
& \quad = \int_{M_t}R_i K(z-z_1)R_i K(\bar z-z_1)dz_1 \int_{M_t} K(z-z_2) K(\bar z-z_2)dz_2\\
& \quad = \int_{M_t}R_i K((z-\bar z)+z_1)R_i K(z_1)dz_1 \int_{M_t} K((z-\bar z)+z_2) K(z_2)dz_2\\
& \quad = (R_i K\ast R_i K)(\bar z-z)(- K\ast K)(\bar z-z),
\end{align*}
where again $M_t=(0,t)\times \T^2$. To get to the last line, we have used that the fractional heat kernel $K$ is even and that we have chosen the mollifier $\rho$ to be even as well. At this point, we note that the decomposition of the heat kernel in the sense of Remark \ref{RemarkDecompKernel} yields two new kernels both being even, This is due to \cite[Lemma 5.24]{Hairer_2014}.\\
Since $ K$ and $R_i K$ are kernels of order $-2$, applying Lemma \ref{LemConvolutionKernels} yields that $R_i K_\varepsilon\ast R_i K$ is a kernel of order $-2+2 \mu$ and  $(R_i K\ast R_i K)(- K\ast K)$ is a kernel of order $-4+4 \mu=2|{\color{blue}\RS{{4i}{r}}}|+4\kappa$ only depending on $(t-\bar t,x-\bar x)$.
Choosing $\nu=2\kappa$ we get (\ref{EstimateW}) and (\ref{EstimateWTime}) by Corollary \ref{corollaryEstimateW}.\\
Next, we show the validity of the estimate (\ref{EstimateWDiff}) and (\ref{EstimateWDiffTime}).  We get
\begin{align*}
& \left\langle\left(\delta\mathcal{W}^{(2)}{\color{blue}\RS{{4i}{r}}}\right)(z),\left(\delta\mathcal{W}^{(2)}{\color{blue}\RS{{4i}{r}}}\right)(\bar z)\right\rangle\\
& \quad = \int_{M_t}\int_{M_t}  \big(R_i K(z-z_1) K(z-z_2)-R_i K(z-z_1)\nabla K(z-z_2)\big)\\
& \qquad\qquad\qquad \big(R_i K(\bar z-z_1) K(\bar z-z_2)-R_i K(\bar z-z_1) K(\bar z-z_2)\big) dz_1dz_2.
\end{align*}
In view of (\ref{EstimateRegularKernelDifference}) in Lemma \ref{LemRegularKernels}, we aim to estimate the differences of $R_i K_\varepsilon-R_i K$ and $K_\varepsilon - K$. Hence, using the algebraic identity $ab-cd=(a-c)(b-d)+(a-c)d+c(b-d)$, we write
\begin{align*}
& \left\langle\left(\delta\mathcal{W}^{(2;\varepsilon)}{\color{blue}\RS{{4i}{r}}}\right)(z),\left(\delta\mathcal{W}^{(2;\varepsilon)}{\color{blue}\RS{{4i}{r}}}\right)(\bar z)\right\rangle\\
& \quad = \bigg[ \int_{M_t}\big(R_i K_\varepsilon(z-z_1)-R_i K(z-z_1)\big)\big(R_i K_\varepsilon(\bar z-z_1)-R_i K(\bar z-z_1)\big)dz_1\\
& \qquad\quad \int_{M_t}\big( K_\varepsilon(z-z_2)- K(z-z_2)\big)\big( K_\varepsilon(\bar z-z_2)- K(\bar z-z_2)\big)dz_2 \bigg]\\
& \qquad\quad + \sum_{i=2}^9 I_i(z,\bar z)\\
& \quad\eqqcolon I_1(z,\bar z) + \sum_{i=2}^9 I_i(z,\bar z),
\end{align*}
where each $I_i(z,\bar z)$, $i=2,...,9$, is an integral over products of different combinations of the kernels $R_i K,\ R_i K_\varepsilon,\  K,\  K_\varepsilon$ and, most importantly, at least one factor is a difference of the form $R_i K_\varepsilon-R_i K$ or $ K_\varepsilon - K$.\\
Now, we can make use of estimate (\ref{EstimateRegularKernelDifference}) and obtain as above
\begin{align*}
\abs{I_1(z,\bar z)} & \leq C \bigg[\int_{M_t}\varepsilon^{\kappa/2}\norm{z-z_1}_\mathfrak{s}^{-2-\kappa/2}\varepsilon^{\kappa/2}\norm{\bar z-z_1}_\mathfrak{s}^{-2-\kappa/2}dz_1\bigg]^2\\
& \underset{\mathclap{Lem.\ref{LemConvolutionKernels}}}{\leq} \quad C \varepsilon^{2\kappa}\big[\norm{\bar z-z}_\mathfrak{s}^{-4+4 \mu -2\kappa}\big]\\
& = C \varepsilon^{\nu'}\norm{\bar z-z}_\mathfrak{s}^{2(|{\color{blue}\RS{{4i}{r}}}|+\kappa)},
\end{align*}
where $\nu'=2\kappa$. Using the same arguments, it is possible to estimate the remaining $I_i(z,\bar z)$ with the same upper bound and the possible exception of different values for $\nu'$ and $\nu$. The fact that one of the factors in $I_i(z,\bar z)$ contains at least one difference of a kernel and its regularized version yields the factor $\varepsilon^{\nu'}$ in the estimate. Since we have only finitely many bounds of the above form, we conclude that the estimate (\ref{EstimateWDiff}) holds by taking $z=(t,x),\bar z=(t,\bar x)$. The estimate (\ref{EstimateWDiffTime}) can be deduced similarly, here we have to handle terms of the form  $I_i((t,x),(t,\bar x))-I_i((t,x),(s,\bar x))-I_i((s,x),(t,\bar x))+I_i((s,x),(s,\bar x))$, which can be done as in the proof of Corollary \ref{corollaryEstimateW}.
Now, by Proposition \ref{PropositionEstimateW} for
\begin{align*}
(\Pi_{x}^t {\color{blue}\RS{{4i}{r}}},\varphi)\coloneqq I_2\left(\int_{M_t}\varphi S_{x}^{\otimes 2}\left(\mathcal{W}^{(2)}{\color{blue}\RS{{4i}{r}}}\right)(z;\cdot)dz\right)
\end{align*}
the bounds (\ref{KeyEstimate}) - (\ref{KeyEstimateDifferenceTime}) hold for ${\color{blue}\RS{{4i}{r}}}$.\\

For the remaining symbols ${\color{blue}\tau}\in\mathcal{F}_-$ all of the estimates follow the same procedure given above for ${\color{blue}\RS{{4i}{r}}}$. First, one establishes (\ref{EstimateW}) and (\ref{EstimateWTime}) with the arguments from above. Then, given the functions $\left(\mathcal{W}^{(k;\varepsilon)}{\color{blue}\tau}\right)(t,x)$, the functions $\left(\mathcal{W}^{(k)}{\color{blue}\tau}\right)(t,x)$ are defined in the same manner as before. Due to the more simpler structure of these functions, the estimates (\ref{EstimateWDiff}) and (\ref{EstimateWDiffTime}) are even easier to obtain. Then, we can define the random variables $(\Pi_{x}^t{\color{blue}\tau},\varphi)$ and conclude estimates (\ref{KeyEstimate}) - (\ref{KeyEstimateDifferenceTime}) using Proposition \ref{PropositionEstimateW}. Thus, the theorem is proved.
\end{proof}


\end{document}